\documentclass{amsart}
\usepackage{amsthm,amsmath,verbatim,amssymb}

\newtheorem{lem}{Lemma}
\newtheorem{defn}[lem]{Definition}
\newtheorem{prop}[lem]{Proposition}
\newtheorem{thm}[lem]{Theorem}
\newtheorem{conj}[lem]{Conjecture}
\newtheorem{cor}[lem]{Corollary}
\newtheorem*{claim}{Claim}

\newcommand{\ddbar}{i\partial\overline{\partial}}
\newcommand{\dbar}{\overline{\partial}}
\renewcommand{\phi}{\varphi}
\renewcommand{\epsilon}{\varepsilon}
\renewcommand{\l}{\mathbf{l}}

\title{On blowing up extremal K\"ahler manifolds}
\author{G\'abor Sz\'ekelyhidi}
\date{}

\begin{document}

\begin{abstract}
	We show that the blowup of
	an extremal K\"ahler manifold at a relatively
	stable point in the sense of GIT admits an
	extremal metric in K\"ahler classes that make the exceptional divisor
	sufficiently small, extending a result of Arezzo-Pacard-Singer. We
	also study the K-polystability of these blowups, sharpening a
	result of Stoppa in this case. As an application we show that 
	the blowup of a K\"ahler-Einstein manifold at a point admits a
	constant scalar curvature K\"ahler metric in classes that make
	the exceptional divisor small, if it is K-polystable with
	respect to these classes. 
\end{abstract}
\maketitle

\section{Introduction}
Let $(M,\omega)$ be a compact K\"ahler manifold of dimension $m$, 
such that $\omega$ is
an extremal metric in the sense of Calabi~\cite{Cal82}. 
This means that the gradient of the scalar curvature
of $\omega$ is a holomorphic vector field, so important special cases
are constant scalar curvature (cscK) metrics and K\"ahler-Einstein
metrics. Following
Arezzo-Pacard~\cite{AP06,AP09} and
Arezzo-Pacard-Singer~\cite{APS06}
we study the problem of constructing extremal
metrics on the blowup of $M$ in one or more points, in K\"ahler classes
which make the exceptional divisors sufficiently small. To state the
result precisely, we make a few definitions.  
The condition that $\omega$ is extremal
implies that the Hamiltonian vector field $X_\mathbf{s}$ corresponding
to the scalar curvature $\mathbf{s}(\omega)$ is a Killing field.
Let $G$ be the group of Hamiltonian
isometries of $(M,\omega)$, 
and write $\mathfrak{g}$ for its Lie
algebra. We fix a moment map
\[ \mu : M \to \mathfrak{g}^*\]
for the action of $G$ on $M$, such that for any vector field 
$X\in \mathfrak{g}$ the function $\langle \mu, X\rangle$ has zero mean
on $M$. We will also identify $\mathfrak{g}$ with its dual
$\mathfrak{g}^*$ using the inner product
\[ \langle X,Y\rangle = \int_M \langle \mu, X\rangle\langle \mu,
Y\rangle \omega^n.\]
for $X,Y\in\mathfrak{g}$, so we will think of elements in
$\mathfrak{g}^*$ as vector fields. Our first main result is then as follows.

\begin{thm}\label{thm:main}
	Choose distinct points $p_1,\ldots, p_n\in M$ and numbers
	$a_1,\ldots, a_n>0$ such that the vector fields $X_\mathbf{s}$
	and  $\sum\limits_i a_i^{m-1}\mu(p_i)$
	vanish at the $p_i$. Then there exists $\epsilon_0 > 0$ such that for
	$\epsilon\in(0,\epsilon_0)$ the blowup $Bl_{p_1,\ldots,p_n}M$ admits an
	extremal metric in the K\"ahler class 
	\[ \pi^*[\omega] - \epsilon^2\left(a_1[E_1] +
	\ldots + a_n[E_n]\right),\]
	where $E_i$ are the
	exceptional divisors and $\pi$ is the blowdown map to $M$.
\end{thm}

To compare with the earlier results, we now describe the theorem proved
by Arezzo-Pacard-Singer in \cite{APS06}. As above $(M,\omega)$ is an
extremal K\"ahler manifold. We choose $K$ to be any
group of Hamiltonian isometries of $(M,\omega)$ such that its Lie
algebra $\mathfrak{k}$ contains the vector field $X_\mathbf{s}$. 
Now $G$ is the
group of Hamiltonian isometries commuting with $K$, and $\mathfrak{g}$
is its Lie algebra. 
We define $\mathfrak{g}' = \mathfrak{g}\cap \mathfrak{k}$, and
define $\mathfrak{g}''$ to be the orthogonal complement, so
\[ \mathfrak{g} = \mathfrak{g}' \oplus \mathfrak{g}''.\]
Then the most general result in \cite{APS06} is the following. 
\begin{thm}[Arezzo-Pacard-Singer] \label{thm:APS}
	With notation as above, let $p_1,\ldots, p_n\in M$ be points
	where each vector field in $\mathfrak{k}$ vanishes. Suppose that
	\begin{enumerate}
		\item[(i)](Balancing condition) we choose
			$a_1,\ldots,a_n>0$ such that
			\[ \sum_{j=1}^n a_j^{m-1}\mu(p_j) \in
			\mathfrak{g}'^*,\]
		\item[(ii)](Genericity condition) the projections of
			$\mu(p_1),\ldots,\mu(p_n)$ onto
			$\mathfrak{g}''^*$ span $\mathfrak{g}''^*$,
		\item[(iii)](General position condition) there is no
			nontrivial element of 
			$\mathfrak{g}''$ that vanishes at
			$p_1,\ldots,p_n$. 
	\end{enumerate}
	Then there exists $\epsilon_0>0$ such that for all $\epsilon\in
	(0,\epsilon_0)$ there is a $K$-invariant extremal K\"ahler
	metric on the blowup $Bl_{p_1,\ldots,p_n}M$ whose K\"ahler class
	is
	\[ [\omega] -
	\epsilon^2\left(a_1[E_1]+
	\ldots+a_n[E_n]\right), \]
	where the $E_i$ are the
	exceptional divisors. 
\end{thm}
In addition condition (iii) can be removed if we allow losing control of
the K\"ahler class a bit (see \cite{APS06} for more details). 

Note that since the vector field $X_\mathbf{s}$ and also any vector
field in $\mathfrak{g}'$ is contained in $\mathfrak{k}$, the assumptions
of Theorem~\ref{thm:APS} imply those of Theorem~\ref{thm:main}. In
particular we do not need conditions (ii) and (iii). Although once 
the number of points blown up is large enough the conditions 
(ii) and (iii) are satisfied generically, it is clearly of interest
to obtain results that work for fewer points. We also see that condition
(i) seems to be weakened, but in fact if we choose $K$ to be the largest
possible fixing the points $p_1,\ldots,p_n$ then condition (i) is
equivalent to the vanishing of $\sum a_i^{m-1}\mu(p_i)$ at the points
$p_i$. 

The new ingredient in the proof of Theorem~\ref{thm:main} is fairly
simple so we describe it here briefly, focusing on the case of blowing
up just one point. Starting with an extremal metric on $M$, in
\cite{APS06} the authors try to directly construct an extremal metric on
the blowup $Bl_pM$ in suitable K\"ahler classes, 
whereas we try to solve a slightly more general
equation instead. More precisely for suitably small $\epsilon$ we find
a metric $\omega_{p,\epsilon}$ on $Bl_pM$ in the K\"ahler class
$[\omega]-\epsilon^2[E]$ together with a vector field $h_{p,\epsilon}\in
\mathfrak{g}$ such that if the vector field
$h_{p,\epsilon}$ vanishes at the point $p$, then $\omega_{p,\epsilon}$
is an extremal metric. So the problem becomes to analyse when
$h_{p,\epsilon}$ vanishes at $p$, but this is a finite
dimensional problem. Varying $p$ we obtain a map
\[ h_\epsilon : M \to \mathfrak{g} \]
for small $\epsilon$, and the crucial point is that $h_\epsilon$ is
a perturbation of the moment map $\mu$. Then a
perturbation argument shows that if $\mu(p)$ vanishes at $p$ then
there is a point $q$ in the same orbit of the complexified group $G^c$
as $p$, such that $h_\epsilon(q)$ vanishes at $q$. This means that we
have an extremal metric on the blowup $Bl_qM$, but this is
biholomorphic to the blowup $Bl_pM$ and this concludes the proof. The
actual proof will be slightly different, since for technical reasons we
will work on a suitable submanifold of $M$ instead of all of $M$. The
idea of separating the problem in this way into an infinite dimensional
problem that is easier to solve than the original, together with a
finite dimensional ``obstruction'' problem is well known (see for
example Hong~\cite{Hong02, Hong08} for a similar technique used to
construct constant scalar curvature metrics on ruled manifolds). 

Recently there has been much work on relating the existence of extremal
metrics to algebro-geometric conditions on the underlying complex
manifold. This work is centered around the following conjecture.
\begin{conj}[Yau-Tian-Donaldson]\label{conj:YTD}
	Let $L$ be an ample line bundle over a
	compact complex manifold $M$. Then there exists an extremal
	metric in $c_1(L)$ if and only if the pair $(M,L)$ is relatively
	K-polystable.
\end{conj}
This conjecture goes back to Yau~\cite{Yau93} in the case of
K\"ahler-Einstein metrics, and the first results are due to
Tian~\cite{Tian97}. Donaldson~\cite{Don02} extended the question to the
case of constant scalar curvature metrics and proved the conjecture in
the case of toric surfaces~\cite{Don08_1}. Two recent surveys on this topic
are Thomas~\cite{Thomas06} and Phong-Sturm~\cite{PS08}. 
Note that it is essentially known that if an
extremal metric exists then the manifold is relatively K-polystable
(see Donaldson~\cite{Don05}, Stoppa~\cite{Sto08,SSz09},
Mabuchi~\cite{Mab08}), although this depends on the precise
definition of stability being used. 

A natural problem is to verify the conjecture for the type of blowups we
are considering, building on existence results like
Theorem~\ref{thm:main} and \ref{thm:APS}. 
Let us suppose that $\omega$ is an extremal
metric on $M$, and $\omega\in c_1(L)$ for some ample line bundle $L\to
M$. For simplicity let us focus on the case of
blowing up just one point. We want to characterize the points $p\in
M$ for which the blowup $Bl_pM$ admits an extremal metric in the class
$c_1(\pi^*L-\epsilon E)$ for small rational $\epsilon$. Note that some
partial results on this question for cscK metrics on multiple blowups
were obtained by Arezzo-Pacard~\cite{AP07}. 

The choice of moment map gives a
lifting of the infinitesimal action of $G$ to the line bundle $L$. 
Replacing $L$ by a large power if necessary we can assume that we obtain
a global action of $G$, and moreover we can extend this to an action of
the complexified group $G^c$. 
In this case Theorem~\ref{thm:main} can be
reformulated as follows.

\begin{cor} Suppose $p\in M$ is such that $X_\mathbf{s}$ vanishes at
	$p$, and $p$ is relatively stable for the action of $G^c$
	on $M$ with respect to the polarization $L$. Then $Bl_pM$
	admits an extremal metric in the class $c_1(\pi^*L-\epsilon E)$
	for sufficiently small $\epsilon$. 
\end{cor}

In order to verify Conjecture~\ref{conj:YTD} we need to relate the
stability of $p\in M$ for the action of $G^c$, to relative
K-polystability of $Bl_pM$ with respect to the class $\pi^*L-\epsilon E$
for small $\epsilon$. In this direction, Stoppa~\cite{Sto10} and
Della Vedova~\cite{DV08} showed that if $p$ is relatively
\emph{strictly} unstable for the action of $G^c$, then $Bl_pM$ is not
relatively K-stable with the polarization $\pi^*L-\epsilon E$ for small
$\epsilon$, and so it does not admit an extremal metric in these
classes. The remaining question is what happens when $p$ is strictly
semistable. We now focus on the cscK case, for which we have the
following. 

\begin{thm}\label{thm:stable} Let $(M,L)$ be a polarized manifold and
	fix $p\in M$. Suppose that for all
	sufficiently small rational $\epsilon>0$, the blowup $Bl_pM$ is
	K-polystable with respect to the polarization $\pi^*L-\epsilon
	E$. Then $p\in M$ is stable for the action of $G^c$ with respect
	to the linearization $L_\delta$ for all sufficiently small
	rational 
	$\delta > 0$, where $L_\delta$ is the following, depending on
	the dimension:
	\begin{itemize}
		\item If $m > 2$ then $L_\delta = L+\delta K_M$, where
			$K_M$ is the canonical bundle.
		\item If $m=2$ and $K_X\cdot L \geqslant 0$ then
			$L_\delta = L + \delta K_M$.
		\item If $m=2$ and $K_X\cdot L < 0$ then 
			$L_\delta = L - \delta K_M$. 
	\end{itemize}
\end{thm}

In comparison, Stoppa's result~\cite{Sto10} implies that if $Bl_pM$ is
K-polystable for the polarizations $\pi^*L-\epsilon E$ for sufficiently 
small $\epsilon$, then $p\in M$ is semistable with respect to the
linearization $L$. 
But if $p\in M$ is stable with respect to the linearization $L\pm\delta K_M$
for all sufficiently small $\delta > 0$, then by letting $\delta\to 0$
it follows easily that $p\in M$ is semistable with respect to $L$, so
our result is slightly stronger. In fact we expect
Theorem~\ref{thm:stable} to be
sharp at least when $m>2$. In order to show this, we would need the
following strengthening of Theorem~\ref{thm:main} in the cscK case,
stated for the blowup in just one point.

\begin{conj}\label{conj:stronger}
	Let $\dim M=m > 2$. 
	Suppose that $M$ admits a cscK metric in $c_1(L)$, and let
	$p\in M$. There exist $\delta_0, \epsilon_0>0$ such that if
	$\mu(p) + \delta\Delta\mu(p)=0$ for some $\delta\in(0,\delta_0)$
	then for all $\epsilon\in(0,\epsilon_0)$ the manifold $Bl_pM$
	admits a cscK metric in the K\"ahler class $c_1(\pi^*L -
	\epsilon E)$. If $m=2$ then we can ask for an analogous result
	to hold, just using $\mu(p) \pm \delta\Delta\mu(p)$ with the
	sign in accordance with the signs in Theorem~\ref{thm:stable}.
\end{conj}

Here $\Delta\mu$ is the Laplacian of $\mu$ taken componentwise after
identifying $\mathfrak{g}^*$ with $\mathbf{R}^l$ for some $l$. Note that
$\mu + \delta\Delta\mu$ is a moment map for the action of $G$ on $M$
with respect to the K\"ahler form $\omega-\delta\rho$, where $\rho$ is
the Ricci form of $\omega$, so by the Kempf-Ness theorem we can find a
zero of $\mu+\delta\Delta\mu$ in the $G^c$-orbit of $p$ if and only if
$p$ is stable with respect to the linearization $L+\delta K_M$ (see
Lemma~\ref{lem:LK} in Section~\ref{sec:stability} for this). So
Theorem~\ref{thm:stable} and Conjecture~\ref{conj:stronger} together
imply that
if $M$ admits a cscK metric in $c_1(L)$ and $Bl_pM$ is K-polystable for
the polarization $\pi^*L-\epsilon E$ for some sufficiently small
$\epsilon$, then $Bl_pM$ admits a cscK metric in $c_1(\pi^*L-\epsilon
E)$. At the end of Section~\ref{sec:gluing} we indicate why we expect
this conjecture to hold at least when $m>2$. 

There is one case when Conjecture~\ref{conj:stronger} follows
directly from Theorem~\ref{thm:main}, namely when $(M,\omega)$ is
K\"ahler-Einstein. The only interesting case is when $M$ is Fano, since
otherwise $M$ does not admit Hamiltonian holomorphic vector fields, so
the blowup at any point admits a cscK metric by the theorem in \cite{AP06}. 
In the Fano 
case we can scale the metric so that 
$\rho=\omega$, in which case $\Delta\mu=-\mu$. The statement of the
conjecture then reduces to that of Theorem~\ref{thm:main}, so we obtain the
following.
\begin{cor}\label{cor:KE}
	Suppose that $(M,\omega)$ is Fano and K\"ahler-Einstein, 
	and let $p\in M$. For sufficiently small rational $\epsilon>0$ the
	following are equivalent:
	\begin{itemize}
		\item[(i)] 
			The blowup $Bl_pM$ admits a cscK metric in the class
			$\pi^*[\omega]-\epsilon[E]$,
		\item[(ii)] $Bl_pM$ is K-polystable with  
			respect to the polarization $\pi^*K_M^{-1} -
			\epsilon E$, for test-configurations that are
			invariant under a maximal torus. 
		\item[(iii)]
			The point $p\in M$ is GIT polystable for the action of
			the automorphism group of $M$, with respect to
			the polarization $K_M^{-1}$. 
	\end{itemize}
\end{cor}
In the second statement we need to use test-configurations invariant
under a 
maximal torus for the implication (i)$\Rightarrow$(ii)
because in~\cite{SSz09} we were not yet able to prove full
K-polystability assuming the existence of a cscK metric.

The outline of the paper is as follows. We first give the proof of the
finite dimensional perturbation result in
Section~\ref{sec:deform} together with a related result in geometric
invariant theory. In Section~\ref{sec:prelim} we discuss some
background material on extremal metrics. 
Then in Section~\ref{sec:mainargument} we
set up the equation that we want to solve and give the proof of
Theorem~\ref{thm:main} assuming that we can solve this equation. 
In Section~\ref{sec:gluing} we solve the equation using a gluing
method similar to \cite{APS06}, completing the proof of
Theorem~\ref{thm:main}. Our 
approach is slightly different than that of Arezzo-Pacard-Singer, 
but the technical ingredients 
are more or less the same. We could also
have adapted the proof in \cite{APS06} more directly to our slightly
more general setting. Finally in Section~\ref{sec:stability} we discuss
the algebro-geometric side of the problem, and we prove
Theorem~\ref{thm:stable} and Corollary~\ref{cor:KE}.

\subsection*{Acknowledgements} I would like to thank D. H. Phong for his
encouragement and interest in this work, and also V. Tosatti for helpful
comments and discussions. This work was partially supported by NSF grant
DMS-0904223.

\section{Deforming relatively stable points in GIT}\label{sec:deform}
Let $U$ be a K\"ahler manifold (open, or compact without boundary)
with K\"ahler form $\omega$, and suppose
that a compact group $H$ acts on $U$, preserving $\omega$. Let $H^c$ be
the complexification of $H$, and $\mathfrak{h}$ be the Lie algebra of
$H$. The action of $H$ extends to a partial action of $H^c$, ie. if $x\in
U$ and $\xi\in\mathfrak{h}$ is sufficiently small then $e^{i\xi}x\in U$.
For any $x\in U$ write $\mathfrak{h}_x$ for the stabilizer of $x$. Also,
let 
\[ \mu: U\to \mathfrak{h} \]
be a moment map for the $H$ action on $U$, where we have identified
$\mathfrak{h}$ with its dual using an invariant inner product. 

\begin{prop}\label{prop:deform}
	Suppose that $x\in U$ satisfies $\mu(x)\in \mathfrak{h}_x$. Let
	$\mu_\epsilon: U\to\mathfrak{h}$ be a family of maps
	such that $\mu_\epsilon\to\mu$ in $C^0$ as $\epsilon\to 0$, and 
	for each $y\in U$ and $\epsilon >0$ 
	the element $\mu_\epsilon(y)$ commutes with the
	stabilizer $\mathfrak{h}_y$. Then for
	sufficiently small $\epsilon$ there exists $\xi\in\mathfrak{h}$
	such that $y=e^{i\xi}x$ satisfies
	$\mu_\epsilon(y)\in\mathfrak{h}_y$.
\end{prop}
\begin{proof}
	This is essentially an application of the implicit function
	theorem.
	Let us first treat the case when $\mu(x)=0$. For every $y\in U$
	write $\mathfrak{h}_y^\perp$ for the orthogonal complement of
	$\mathfrak{h}_y$, and let $\Pi_y$ be the orthogonal projection
	onto $\mathfrak{h}_y^\perp$. For $\xi$ in a small ball
	$B\subset\mathfrak{h}_x^\perp$ we define the projection
	$P_\xi:\mathfrak{h}\to\mathfrak{h}_x^\perp$ by
	\[ P_\xi(\eta) = \Pi_x\Pi_{e^{i\xi}x}(\eta).\]
	Then $P_\xi(\mu(e^{i\xi}x))=0$ means that $\mu(e^{i\xi}x)\in
	\mathfrak{h}_{e^{i\xi}x}$.
	For small $\epsilon$ define the map $F_\epsilon$ by
	\[ \begin{aligned}
		F_\epsilon : B\subset\mathfrak{h}_x^\perp &\to
		\mathfrak{h}_x^\perp \\
		\xi &\mapsto P_\xi(\mu_\epsilon(e^{i\xi}x)).
	\end{aligned}\]
	Then $F_0(0)=0$ and $DF_0$ at $0$ is just
	the derivative of $\mu$ at $x$ (we use here that $\mu(x)=0$).
	Since this is an isomorphisms
	$\mathfrak{h}_x^\perp\to\mathfrak{h}_x^\perp$, we know that 
	$F_0(B)$
	contains a small ball $B_\delta$ around the origin. It then
	follows from degree considerations that
	for sufficiently small $\epsilon$, the image
	$F_\epsilon(B)$ contains a small ball $B_{\delta'}$ around the
	origin, in particular $F_\epsilon(\xi)=0$ for some $\xi$. But
	this means that $y=e^{i\xi}x$ satisfies
	$\mu(y)\in\mathfrak{h}_y$.
	
	If $\mu(x)=\xi\not=0$, then we reduce to the previous case as
	follows. Let $T\subset H$ be the closure of the subgroup of $H$
	generated by $\xi$, so $T$ is a torus. Write $H_T$ for the
	centraliser of $T$ in $H$, and $\mathfrak{h}_T$ for its Lie
	algebra. Inside this Lie algebra let $\mathfrak{h}_{T^\perp}$ be
	the orthogonal complement of the Lie algebra of $T$, and let
	$H_{T^\perp}$ be the corresponding subgroup of $H$. We then look
	at the action of $H_{T^\perp}$ on $U$, for which the moment map
	$\mu_T$ is
	simply the projection of $\mu$ onto $\mathfrak{h}_{T^\perp}$. But
	then $\mu_T(x)=0$, so we can apply the 
	previous argument to find
	$\xi\in\mathfrak{h}_{T^\perp}$ such that $y=e^{i\xi}x$ satisfies
	$\mu_{\epsilon,T}(y)\in\mathfrak{h}_y$. Here
	$\mu_{\epsilon,T}$ is the projection of $\mu_\epsilon$ onto
	$\mathfrak{h}_{T^\perp}$. Since $\mu_\epsilon(y)$ commutes with
	elements in the stabiliser $\mathfrak{h}_y$, in particular it
	commutes with $\mathfrak{t}$. Hence $\mu_\epsilon(y)$ differs from the
	projection $\mu_{\epsilon,T}(y)$ by an element in $\mathfrak{t}$, 
	so $\mu_\epsilon(y)\in\mathfrak{h}_y$.
\end{proof}

It is helpful to put this
result into the context of relative stability. In this setting $U$ is a
compact K\"ahler manifold and the symplectic form $\omega$ represents
the first Chern class of a $\mathbf{Q}$-line bundle $L$ over $U$.
Moreover the choice of moment map $\mu$ corresponds to a choice of
lifting of the action to some power of
$L$, called a choice of linearization. 

\begin{defn}
	A point $x\in U$ is \emph{relatively stable} if there exists a
	point $y$ in the $H^c$-orbit of $x$ for which
	$\mu(y)\in\mathfrak{h}_y$. 
\end{defn}

The relationship of this definition using moment maps to geometric
invariant theory is a version of the Kempf-Ness theorem~\cite{KN79} and is
worked out in~\cite{GSz04} (see also
Kirwan~\cite{Kir84}). Using this terminology,
Proposition~\ref{prop:deform} says that
if $x$ is relatively stable for a certain choice of line bundle and
linearization then it is still relatively stable for small
perturbations. It is more general however, because we do not need to
know that the $\mu_\epsilon$ are also moment maps. 

In the rest of this 
section we study what more we can say about the stability of points
in the sense of GIT as we deform the polarization. We will only consider
a very simple kind of deformation, namely we have two line bundles $L$
and $K$ on $M$, such that $L$ is ample. We suppose that a complex reductive
group $G$ acts on $M$ and we choose linearizations of the action on $L$ and
$K$. Fix a point $p\in M$, and let $\lambda$ be a one-parameter subgroup in
$G$. Let
\[ q = \lim_{t\to 0}\lambda(t)p.\]
Then $\lambda$ fixes the point $q$ and we write $w_L(p,\lambda)$ for the
weight of the action of $\lambda$ on the fiber $L_q$, and $w_K(p,\lambda)$
for the weight on $K_q$. The following is well-known, see
Mumford-Fogarty-Kirwan~\cite{MFK94}.

\begin{prop}[Hilbert-Mumford criterion] The point $p\in M$ is semistable
	with respect to the polarization $L$ if and only if
	$w_L(p,\lambda) \geqslant 0$ for all one-parameter subgroups
	$\lambda$. If $w_L(p,\lambda) > 0$ for all $\lambda$ which does
	not fix $p$, then $p$ is polystable.
\end{prop}

The result we will need is the following version of this for the deformed
polarization $L+\epsilon K$. 
\begin{prop}\label{prop:HM}
	There exists $\epsilon_0>0$ such that for all
	$\epsilon\in(0,\epsilon_0)$ the point $p$ is polystable with respect
	to the polarization $L+\epsilon K$ if and only if $p$ is
	semistable with respect to $L$ and for every one-parameter
	subgroup $\lambda$ for which $w_L(p,\lambda)=0$ and $\lambda$
	does not fix $p$, 
	we have $w_K(p,\lambda) > 0$. In
	particular whether $p$ is polystable with respect to
	$L+\epsilon K$ or not is independent of the choice of
	$\epsilon\in(0,\epsilon_0)$.
\end{prop}
\begin{proof}
	First let us assume that $K$ is ample.  
	Choose a maximal torus $T\subset G$ and let $\mathfrak{t}$ be
	its Lie algebra. 
	We will use Proposition 2.14 in \cite{MFK94}. This says that
	there are a finite number of rational linear functionals $l_i,
	m_j\in \mathfrak{t}^*$, such that for every point $x\in M$ and
	one-parameter subgroup $\lambda\subset T$ we have
	\[ \begin{aligned}
		w_L(x,\lambda) &= \max\{ l_i(\lambda)\,|\, i\in I(x)\},
		\\
		w_K(x, \lambda) &= \max\{ m_j(\lambda)\, |\, j\in J(x)\}.
	\end{aligned}\]
	Here we identified $\lambda$ with its generator in
	$\mathfrak{t}$ and $I(x)$, $J(x)$ are finite index sets
	depending on $x$. 

	If we have a one-parameter subgroup $\lambda\subset G$ which is
	not in $T$, then we can always find a conjugate
	$\gamma\lambda\gamma^{-1}\subset T$, and we have
	\[ \begin{aligned}
		w_L(x,\lambda) &= w_L(\gamma x,
		\gamma\lambda\gamma^{-1})= \max\{
		l_i(\gamma\lambda\gamma^{-1})\,|\, i\in I(\gamma x)\},
		\\
		w_K(x, \lambda) &= w_K(\gamma x,
		\gamma\lambda\gamma^{-1}) = \max\{
		m_j(\gamma\lambda\gamma^{-1})\, |\, j\in J(\gamma x)\}.
	\end{aligned}\]
	We want to show that for sufficiently small $\epsilon$, the
	weight $w_L(p,\lambda) + \epsilon w_K(p,\lambda)$ is positive
	for all $\lambda$ which does not fix $p$. It is enough to check
	this for $\lambda\subset T$, because allowing conjugate
	one-parameter subgroups $\gamma\lambda\gamma^{-1}$ amounts to
	replacing $p$ by $\gamma p$, but in the weight computation 
	all that matters is the index
	sets $I(\gamma p)$ and $J(\gamma p)$. Since there are only
	finitely many of these, if we find an $\epsilon$ that works for
	each case separately, then we can take the minimum of these. 

	Restricting attention to one-parameter subgroups $\lambda\subset
	T$, we can extend the definition of $w_L(x,\lambda)$
	continuously to $\lambda\in\mathfrak{t}_\mathbf{R}$. By taking a
	smaller torus we can also assume that no element in $T$ fixes
	$p$. The main
	point is that then the set of $\lambda$ for which
	$w_L(p,\lambda)=0$ is a convex cone 
	$\mathcal{C}\subset\mathfrak{t}_\mathbf{R}$, 
	whose extremal rays are rational. By our assumption
	$w_K(p,\lambda) > 0$ for $\lambda\in \mathcal{C}\cap
	\mathfrak{t}_\mathbf{Q}$, but then this is true for
	all $\lambda\in \mathcal{C}$ because of the rationality of the
	$m_j$. Let us write $\partial B\subset\mathfrak{t}$ 
	for the unit sphere
	with respect to some fixed norm. We
	then have $w_K(p,\lambda) > 0$ on $\mathcal{C}\cap \partial B$, 
	but the latter is compact, so there exists an open 
	neighbourhood $U\subset \partial B$ of $\mathcal{C}\cap \partial
	B$ for which
	\[ w_K(p,\lambda) > 0\,\quad\text{ for }\lambda\in
	U.\]
	At the same time $|w_K(p,\lambda)|<C$ for some constant $C$ and
	all $\lambda\in \partial B$. In addition there exists $\delta>0$ such
	that $w_L(p,\lambda) > \delta$ for $\lambda\not\in U$.
	Finally it follows that if $\epsilon < \delta/C$, then  
	\[ w_L(p,\lambda) + \epsilon w_K(p,\lambda) > 0\]
	for all $\lambda$. 

	For the converse direction, 
	we note that $|w_K(p,\lambda)| <
	C$ for some constant $C$ and all $\lambda\in \partial B$, with $C$
	independent of $p$. If $p$ is not
	semistable with respect to $L$, then there is a one-parameter
	subgroup $\lambda$ for which $w_L(p,\lambda) = -\delta < 0$. But
	then for all $\epsilon < \delta / (C|\lambda|)$ we have
	\[ w_L(p,\lambda) + \epsilon w_K(p,\lambda) < -\delta +
	C\epsilon|\lambda| < 0,\]
	so $p$ is not stable with respect to $L+\epsilon K$ for any
	sufficiently small $\epsilon$. The fact that we need
	$w_K(p,\lambda) > 0$ for all $\lambda$ such that
	$w_L(p,\lambda)=0$ is immediate. 

	Now suppose that $K$ is not ample. We
	can choose a large constant $c$ such that $cL+K$ is ample. Then
	note that for small $\epsilon$
	\[ L + \epsilon K = 
	(1-\epsilon c)\left( L + \frac{\epsilon}{1-\epsilon
	c}(cL+K)\right),\]
	so if $p$ is polystable
	with respect to $L+\epsilon K$ then is is also polystable
	with respect to $L+\frac{\epsilon}{1-\epsilon c}(cL +
	K)$, where $cL+K$ is ample. If $\epsilon$ is sufficiently small,
	then we can apply what we just proved. So $p$ is polystable with
	respect to $L+\epsilon K$ if and only if $p$ is semistable with
	respect to $L$ and for every one-parameter
	subgroup $\lambda$ for which $w_L(p,\lambda)=0$ and $\lambda$
	does not fix $p$, we have $w_K(p,\lambda) =
	w_{cL+K}(p,\lambda)>0$.
\end{proof}

\section{Background on extremal metrics} \label{sec:prelim}
In this section we collect some material which we will need later on. 

\subsection{The extremal metric equation}
As we said before, the basic strategy to obtain an extremal metric on a
blowup $Bl_pM$ is to first use the extremal metric $\omega$ 
on $M$ and a simple
gluing argument to obtain an approximately extremal metric
$\omega_\epsilon$ on $Bl_pM$ and then to try perturbing this to an
extremal metric. When we set up the problem more precisely 
later we will have a
maximal torus of automoprhisms $T$ acting on $Bl_pM$, preserving the
approximate solution $\omega_\epsilon$ and we will seek an extremal
metric of the form 
\[ \omega_\epsilon + \ddbar\phi,\]
where $\phi$ is $T$-invariant. Let us write
$\overline{\mathfrak{t}} \subset C^\infty(Bl_pM)$ for the space of
Hamiltonian functions generating elements of $T$, which includes the
constants. Note that $\dim\overline{\mathfrak{t}}=\dim\mathfrak{t}+1$ where
$\mathfrak{t}$ is the Lie algebra of $T$, and $\overline{\mathfrak{t}}$
consists of the smooth $T$-invariant functions in the kernel of the
Lichnerowicz operator on $Bl_pM$, defined in Section~\ref{sec:Lichnerowicz}.
We need the following which can also be found in
\cite{APS06}.
\begin{lem}\label{lem:basiceqn}
	Suppose that $\phi\in C^\infty(Bl_pM)^T$ and 
	$f\in\overline{\mathfrak{t}}$ such that
	\begin{equation}
		\label{eq:basiceqn}
		\mathbf{s}
		(\omega_\epsilon+\ddbar\phi) - \frac{1}{2}\nabla
		f\cdot\nabla\phi = f,
	\end{equation}
	where the gradient and inner product are computed with respect
	to the
	metric $\omega_\epsilon$.
	Then $\omega_\epsilon + \ddbar\phi$ is an extremal metric. 
\end{lem}
\begin{proof}
	Let $X$ be the holomorphic vector field on $Bl_pM$ with
	Hamiltonian function $f$, ie. 
	\[ df = \iota_X\omega_\epsilon.\]
	At the same time we can compute 
	\[ \iota_X(\ddbar \phi) = \frac{1}{2} d(JX(\phi)). \]
	Since $JX=\nabla f$, by
	combining the previous two formulae we get
	\[ 
		\iota_X(\omega_\epsilon + \ddbar\phi) = d\left(f +
		\frac{1}{2} \nabla f\cdot\nabla\phi\right) =
		d\mathbf{s}(\omega_\epsilon+\ddbar\phi).
	\]
	This means that $\omega_\epsilon + \ddbar\phi$ is an extremal
	metric.
\end{proof}

In order to solve Equation~(\ref{eq:basiceqn}) as a perturbation problem,
we will write it in the form
\begin{equation}\label{eq:2}
	\mathbf{s}(\omega_\epsilon+\ddbar\phi) -
	\frac{1}{2}\nabla(\tilde{s}+\tilde{f})\cdot\nabla\phi =
	\tilde{s}+\tilde{f},
\end{equation}
where $\tilde{s},\tilde{f}\in \overline{\mathfrak{t}}$, and
$\tilde{s}$ is chosen so that
the holomorphic vector field $\nabla\tilde{s}$ is the
natural holomorphic lift of the vector field $\nabla \mathbf{s}(\omega)$ 
on $M$.
In addition we can normalise $\tilde{s}$ so that it agrees with
$\mathbf{s}(\omega)$ outside a small ball around $p$, 
where the metrics $\omega$
and $\omega_\epsilon$ coincide. The advantage of this is that we now
seek $\phi$ and $\tilde{f}$ which are small, or in other words, setting
$\phi=0$ and $\tilde{f}=0$ we get an approximate solution to the
equation. 

For any metric $\tilde{\omega}$ let us define the operators
$L_{\tilde{\omega}}$ and $Q_{\tilde{\omega}}$ by
\begin{equation}\label{eq:scal}
	\mathbf{s}(\tilde{\omega} + \ddbar\phi) = \mathbf{s}
	(\tilde{\omega}) +
	L_{\tilde{\omega}}(\phi) + Q_{\tilde{\omega}}(\phi),
\end{equation}
where $L$ is the linearized operator. A simple computation shows that
\[ L_{\tilde{\omega}}(\phi) = \Delta^2_{\tilde{\omega}}\phi +
\mathrm{Ric}(\tilde{\omega})^{i\bar j}\phi_{i\bar j},\]
and analysing this operator will be crucial later on. We are using the
complex Laplacian here which is half of the usual Riemannian one.  

At the same time note that the
linear operator appearing in the linearization of Equation~(\ref{eq:2}) is
\begin{equation}
	\label{eq:linop}
	(\phi,\tilde{f}) \mapsto  L_{\omega_\epsilon}(\phi) -
	\frac{1}{2}\nabla\tilde{s}\cdot 
	\nabla \phi - \tilde{f},
\end{equation}
which is closely related to the Lichnerowicz operator. 

\subsection{The Lichnerowicz operator}\label{sec:Lichnerowicz}
For any K\"ahler metric
$\tilde{\omega}$ on a manifold $X$ we have the operator
\[ \mathcal{D}_{\tilde{\omega}}:C^\infty(X) \to \Omega^{0,1}(T^{1,0}X),\]
given by $\mathcal{D}(\phi) = \dbar\nabla^{1,0}\phi$ where $\dbar$ is
the natural $\dbar$-operator on the holomorphic tangent bundle. 
The Lichnerowicz operator is then the fourth order operator
\[ \mathcal{D}^*_{\tilde{\omega}}\mathcal{D}_{\tilde{\omega}} 
: C^\infty(X)\to C^\infty(X),\]
whose significance is that the kernel consists of precisely those 
functions whose gradients are holomorphic vector fields.
The relation to the operator in Equation~(\ref{eq:linop}) is that a
computation (see eg. LeBrun-Simanca~\cite{LS94}) shows that
\begin{equation}\label{eq:Lichn}
	\mathcal{D}^*_{\tilde{\omega}}\mathcal{D}_{\tilde{\omega}}(\phi)
	= L_{\tilde{\omega}}(\phi) - \frac{1}{2}\nabla
	\mathbf{s}(\tilde{\omega})\cdot\nabla\phi,
\end{equation}
but note that in general $\mathbf{s}(\tilde{\omega})$ is not equal to
$\tilde{s}$.

\subsection{Burns-Simanca metric}\label{sec:BurnsSimanca}
The approximate metric
$\omega_\epsilon$ on $Bl_pM$ is constructed by gluing the extremal
metric $\omega$ on $M$ to a rescaling of a 
suitable model metric on $Bl_0\mathbf{C}^m$,
ie. on the blowup of $\mathbf{C}^m$ at the origin. This model metric is
a scalar flat metric found by Burns (see LeBrun~\cite{LeBrun88})
for $m=2$ and by
Simanca~\cite{Sim91}
for
$m\geqslant 3$. Away from the exceptional divisor it can be written in
the form
\[ \eta = \ddbar\left( \frac{1}{2}|z|^2 +\psi(z) \right),\]
where $z=(z_1,\ldots,z_m)$ are standard coordinates on $\mathbf{C}^m$.
For $m=2$ we have $\psi(z) = \log|z|$ while for $m>2$ we have 
\[ \psi(z) = -|z|^{4-2m} + O(|z|^{3-2m}) \]
for large $|z|$. The quantity $O(|z|^{3-2m})$ is a function in the space
$C^{k,\alpha}_{3-2m}(Bl_0\mathbf{C}^m)$ in the notation of
section~\ref{sec:gluing}, for any $k$ and $\alpha\in(0,1)$. See
Lemma~\ref{lem:BSmetric} for a sharper asymptotic expansion.

\section{The main argument}\label{sec:mainargument}
Suppose as before that $\omega$ is an extremal K\"ahler metric on $M$.
Let $X_\mathbf{s}$ be the Hamiltonian vector field corresponding to the
scalar curvature $\mathbf{s}(\omega)$.  
Write $G$ for the Hamiltonian isometry group of
$(M,\omega)$, so the Lie algebra $\mathfrak{g}$ of $G$ consists
of holomorphic Killing fields with zeros. 

Choose a point $p\in M$ where the vector
field $X_\mathbf{s}$ vanishes, and let $T\subset G$ be a maximal torus
of the subgroup fixing $p$. Let $H\subset G$ consist of the elements
commuting with $T$ and let us write $\overline{\mathfrak{h}}
\subset C^{\infty}(M)$
for the space of Hamiltonian functions of vector fields in the Lie
algebra of $H$. Note that $\overline{\mathfrak{h}}$ contains
the constants as well. 
Let us also write $\overline{\mathfrak{t}}\subset\overline{\mathfrak{h}}$ 
for the Hamiltonian functions corresponding to the subgroup $T\subset
H$. 

Given a small parameter $\epsilon > 0$,
we will construct an approximate solution to our problem 
on $Bl_pM$ in the K\"ahler class $[\omega] - \epsilon^2d_m[E]$, for some
constant $d_m$ depending on the dimension, so that $d_m^{m-1}$ is the
volume of the exceptional divisor of $Bl_0\mathbf{C}^m$ with the
Burns-Simanca metric $\eta$ from Section~\ref{sec:BurnsSimanca}. Of
course we could make $d_m=1$ by rescaling $\eta$. 
For simplicity assume that the exponential map is defined on the unit
ball in the tangent space $T_pM$ (if not, we can scale up the metric
$\omega$). Choose local normal
coordinates $z$ near $p$ such that the group $T$
acts by unitary transformations on the unit ball $B_1$ around $p$ (this
is possible by linearizing the action, see Bochner-Martin~\cite{BM48}
Theorem 8). In
these coordinates we can write
\[ \omega = \ddbar\big(|z|^2/2 + \phi(z)\big),\]
where $\phi=O(|z|^4)$. At the same time recall
the Burns-Simanca metric
\[ \eta = \ddbar\big(|z|^2/2 + \psi(z)\big).\]
We glue $\epsilon^2\eta$ to $\omega$ using a cutoff function in the
annulus $B_{2r_\epsilon}\setminus B_{r_\epsilon}$ in $M$, where the
dependence of $r_\epsilon$ on $\epsilon$ will be chosen later. 
To do this, let $\gamma:\mathbf{R}\to[0,1]$ be smooth such
that $\gamma(x)=0$ for $x < 1$ and $\gamma(x)=1$ for $x > 2$ and then 
define
\[ \gamma_1(r) = \gamma(r/r_\epsilon), \] and
write $\gamma_2=1-\gamma_1$. Then we can define a K\"ahler metric
$\omega_\epsilon$ on $Bl_pM$ which on the annulus $B_1\setminus
B_\epsilon$ is given by 
\[
	\omega_\epsilon = \ddbar\left( \frac{|z|^2}{2} 
	+ \gamma_1(|z|)\phi(z) 
+ \gamma_2(|z|) \epsilon^2\psi(\epsilon^{-1}
z)\right).\]
Moreover outside $B_{2r_\epsilon}$ the metric $\omega_\epsilon=\omega$ while
inside the ball $B_{r_\epsilon}$ we have
$\omega_\epsilon=\epsilon^2\eta$. 
Note that the action of $T$ lifts to $Bl_pM$ 
giving biholomorphisms,
and that $\omega_\epsilon$ is $T$-invariant.

It will be important to lift functions in $\overline{\mathfrak{h}}$ to 
$Bl_pM$. Only elements in
$\overline{\mathfrak{t}}$ have a natural lifting, since they correspond to
holomorphic vector fields vanishing at $p$, so we give the following
definition.

\begin{defn}\label{def:lifting}
	We define a linear map
	\[ \l : \overline{\mathfrak{h}} \to C^\infty(Bl_pM) \]
	as follows. 
	First let us decompose $\overline{\mathfrak{h}}$ into a direct sum
	$\overline{\mathfrak{h}}=\overline{\mathfrak{t}}\oplus\mathfrak{h}'$,
	where we can 
	assume that each function in $\mathfrak{h}'$ vanishes at $p$. Any
	$f\in\overline{\mathfrak{t}}$ 
	corresponds to a holomorphic vector field $X_f$ on
$M$ vanishing at $p$. 
For such $f$ we define $\l(f)$ to be the Hamiltonian
function of the holomorphic lift of the vector field $X_f$ to
$Bl_pM$,
with respect to the symplectic form $\omega_\epsilon$, normalized so
that $f=\l(f)$ outside $B_1$. For
$f\in\mathfrak{h}'$ we simply let $\l(f)=\gamma_1 f$ near $p$
using the cutoff function $\gamma_1$ 
from before, and we think of this $\l(f)$
as a function on $Bl_pM$. Finally define the lift
of general elements in $\overline{\mathfrak{h}}$ by linearity.
\end{defn}

We can now state the main technical result we need, whose proof will be
given in Section~\ref{sec:gluing}.

\begin{prop}\label{prop:gluing}
	Suppose that the point $p\in M$ is chosen so that
	the vector field $X_\mathbf{s}$ vanishes at $p$. 
	Then there are constants $\epsilon_0, c > 0$ such that 
	for all
	$\epsilon\in (0,\epsilon_0)$ we can find $u\in
	C^\infty(Bl_pM)^T$ and $f\in\overline{\mathfrak{h}}$
	satisfying the equation
	\begin{equation}\label{eq:glued}
		\mathbf{s}
		(\omega_\epsilon+\ddbar u) - \frac{1}{2}\nabla
		\l(f)\cdot\nabla u = \l(f).
	\end{equation}
	In addition the element $f\in\overline{\mathfrak{h}}$ has
	an expansion
	\begin{equation}
		\label{eq:expansion}
		f = \mathbf{s}  +
		\epsilon^{2m-2}(\lambda+c_m \mu(p)) + f_\epsilon,
	\end{equation}
	where $c_m$ is a constant depending only on the dimension,
	$\lambda = \mathrm{Vol}(M)^{-1}c_m$ is another constant,
	and $|f_\epsilon| \leqslant c\epsilon^\kappa$ for some
	$\kappa > 2m-2$.
\end{prop}

Note that in this proposition, $\l(f)$ corresponds to a
Hamiltonian vector field $X_{\l(f)}$ on $Bl_pM$, and if this vector
field is holomorphic, then the metric $\omega_\epsilon + \ddbar u$ above
is extremal by Lemma~\ref{lem:basiceqn}. Moreover $X_{\l(f)}$ is
holomorphic if and only if $f\in\overline{\mathfrak{t}}$, ie. if the vector
field $X_f$ on $M$ vanishes at $p$. Given this proposition we can now prove
Theorem~\ref{thm:main}.

\begin{proof}[Proof of Theorem~\ref{thm:main}]
	We will give the proof for the blowup of one point to simplify
	the notation, since blowing up several points does not give rise to
	essential new difficulties. 
	Let us use the notation from before, so that $G$ is the
	Hamiltonian isometry group of $(M,\omega)$, $p\in M$ and $T$ is
	a maximal torus in the stabilizer of $p$. The subgroup $H\subset
	G$ consists of the elements commuting with each element of $T$,
	and let $\mathfrak{h}, \mathfrak{t}$ be the Lie algebras of
	$H,T$. Note that in this case $\mathfrak{h}_p=\mathfrak{t}$,
	where $\mathfrak{h}_p$ is the stabilizer of $p$ in
	$\mathfrak{h}$. 

	We will work on the $H^c$-orbit of $p$, so let us write
	$U=H^c\cdot p$. Then $U\subset M$ is an $H$-invariant
	complex submanifold. If $\mu(p)\in\mathfrak{h}_p$, then the
	stabilizer of $p$ in $H^c$ is $T^c$. This can be seen using the
	structure of the stabilizer group of relatively stable points
	(analogous to Calabi's structure theorem for the automorphism
	groups of extremal metrics~\cite{Cal85}). 
	Since every element in $H^c$ commutes with $T$, it follows
	that for every $q\in U$ the stabilizer of $q$ in
	$H$ is $T$. We
	can therefore apply Proposition~\ref{prop:gluing} to each point
	$q\in U$ with $T$ as the maximal torus. We can first replace $U$
	by a relatively compact complex submanifold $U'\subset\subset U$
	which is still $H$-invariant and contains $p$, 
	to ensure that we can choose
	$\epsilon,c$ in the proposition uniformly over $U'$. Note that
	the solution of the equation in Proposition~\ref{prop:gluing} is
	obtained using the contraction mapping theorem, so although the
	solution is not unique, the various choices can be made so that
	it depends smoothly on the data.

	For a suitably small $\epsilon$ we therefore have a smooth map
	\[\begin{aligned}  \mu_\epsilon : U' &\to \mathfrak{h} \\
			\mu_\epsilon(q)&= \mu(q) +
			c_m^{-1}\epsilon^{2-2m}f_\epsilon,
		\end{aligned}\]
	where $f_\epsilon\in \mathfrak{h}$ is given by the
	Proposition \ref{prop:gluing}
	applied at the point $q$ (so $f_\epsilon$ depends on
	$q$). 
	Since $|f_\epsilon|\leqslant c\epsilon^\kappa$ for
	some $\kappa > 2m-2$, it follows that 
	\[ \lim_{\epsilon\to 0}\mu_\epsilon = \mu.\]
	Applying
	Proposition~\ref{prop:deform} we see that if the vector field
	$\mu(p)$ vanishes at the point $p$
	then for sufficiently small $\epsilon > 0$ we can find a
	point $q$ in the $H^c$-orbit of
	$p$ such that $\mu_\epsilon(q)$ vanishes at $q$. 
	Note that $X_\mathbf{s}$ is in the center of $\mathfrak{g}$, so
	if the vector field $X_\mathbf{s}$ vanishes at $p$ then it also
	vanishes at $q$. This means that when applying 
	Proposition~\ref{prop:gluing} at the point $q$, the element
	$f\in\overline{\mathfrak{h}}$ is actually in
	$\overline{\mathfrak{t}}$, ie. $X_f$ vanishes at $q$. 
	By Lemma~\ref{lem:basiceqn} we therefore 
	obtain an extremal metric on the blowup
	$Bl_qM$ in the K\"ahler class
	\[ [\omega] - \epsilon^2d_m[E].\]
	Since $q$ is in the $H^c$-orbit of $p$, the
	manifolds $Bl_qM$ and $Bl_pM$ are
	biholomorphic so we have constructed an extremal metric on the
	blowup $Bl_pM$.  
\end{proof}

\section{The gluing argument}
\label{sec:gluing}
In this section we prove Proposition~\ref{prop:gluing}. As before
we will only blow up one point, but the proof of the general case is
identical apart from more complicated notation. We will
mainly focus on the case $m\geqslant 3$ since the
case $m=2$ needs special care but we will make brief comments on how to
adapt the arguments when $m=2$.
We first need some analytic preliminaries.

\subsection{The Lichnerowicz operator on weighted spaces}
The key to solving our equation using a perturbation method is to
construct an inverse to the linear operator (\ref{eq:linop}) and to
control its inverse acting between suitable Banach spaces. It turns out
that weighted H\"older spaces are suitable spaces to work in and in
order to understand the mapping properties of the operator
(\ref{eq:linop}) between these spaces on the blowup $Bl_pM$ we first
need to understand the behaviour of the Lichnerowicz operator on weighted
spaces on the manifolds $M\setminus\{p\}$ and $Bl_p\mathbf{C}^m$. This
is the fundamental tool in Arezzo-Pacard~\cite{AP06,AP09} and
Arezzo-Pacard-Singer~\cite{APS06} and we follow their treatment here.
See also Lockhart-McOwen~\cite{LM85}, Mazzeo~\cite{Maz91},
Melrose~\cite{Mel93} or Pacard-Rivi\`ere~\cite{PR00} for
more details on weighted spaces.

First we look at $M_p=M\setminus\{p\}$ with the metric $\omega$. For
functions $f:M_p \to\mathbf{R}$ we define the weighted norm
\[ \Vert f\Vert_{C^{k,\alpha}_\delta(M_p)} = \Vert
f\Vert_{C^{k,\alpha}_\omega(M\setminus B_{1/2})}
+ \sup_{r < 1/2} r^{-\delta}
\Vert f\Vert_{C^{k,\alpha}_{r^{-2}\omega}(B_{2r}
\setminus B_r)}.\]
Here the subscripts $\omega$ and $r^{-2}\omega$ indicate the metrics
used for computing the corresponding norm. The weighted space
$C^{k,\alpha}_\delta(M_p)$ consists of functions on
$M\setminus\{p\}$ which are locally in $C^{k,\alpha}$ and whose
$\Vert\cdot\Vert_{C^{k,\alpha}_\delta}$ norm is finite.

The main result we need is the following.
\begin{prop}\label{prop:Mp}
	If $\delta < 0$ and $\alpha\in(0,1)$ then the operator 
	\[ \begin{aligned}
		C^{4,\alpha}_{\delta}(M_p)^T
		\times\overline{\mathfrak{h}}
		&\to C^{0,\alpha}_{\delta-4}(M_p)^T \\
		(\phi, f) &\mapsto
		\mathcal{D}_\omega^*\mathcal{D}_\omega
		(\phi) - f
	\end{aligned}\]
	has a bounded right-inverse. Here $T$ is a torus of isometries
	of $(M,\omega)$ and $\overline{\mathfrak{h}}$ 
	is the space of $T$-invariant
	Hamiltonian functions of holomorphic killing fields.
\end{prop}
\begin{proof}
	This follows from the duality theory in weighted spaces. The
	image of 
	\[ \mathcal{D}^*\mathcal{D} : C^{4,\alpha}_\delta(M_p) \to
	C^{0,\alpha}_{\delta-4}(M_p)\]
	is the orthogonal complement of the kernel of
	\[ \mathcal{D}^*\mathcal{D} : C^{4,\alpha}_{4-2m-\delta}(M_p)
	\to C^{0,\alpha}_{-2m-\delta}(M_p).\]
	If $\delta < 0$ then $4-2m-\delta > 4-2m$, so we need to see
	that if $h\in\mathrm{Ker}\mathcal{D}^*\mathcal{D}$ is such that
	$h\in C^{4,\alpha}_\gamma(M_p)^T$ for some $\gamma > 4-2m$, then
	$h$ is smooth. 
	This follows from the regularity theory in weighted spaces since
	there are no indicial roots in $(4-2m,0)$. 
\end{proof}

Let us turn now to the manifold $Bl_0\mathbf{C}^m$ with the
Burns-Simanca metric $\eta$. The relevant weighted H\"older norm is now
given by
\[ \Vert f\Vert_{C^{k,\alpha}_\delta(Bl_0\mathbf{C}^m)} = \Vert
f\Vert_{C^{k,\alpha}_\eta(B_2)} + \sup_{r > 1} r^{-\delta} \Vert
f\Vert_{C^{k,\alpha}_{r^{-2}\eta}(B_{2r}\setminus B_r)}.\]
Here we abused notation slightly by writing $B_r\subset Bl_0
\mathbf{C}^m$ for the set where $|z|<r$ (ie. the pullback of the
$r$-ball in $\mathbf{C}^m$ under the blowdown map).

The key result here is the following.
\begin{prop}\label{prop:BlC}
	For $\delta > 4-2m$ and $\alpha\in(0,1)$ the operator
	\[ \begin{aligned}
		C^{4,\alpha}_\delta(Bl_0\mathbf{C}^m)^T &\to
		C^{0,\alpha}_{\delta-4}(Bl_0\mathbf{C}^m)^T \\
		\phi &\mapsto \mathcal{D}^*_\eta\mathcal{D}_\eta(\phi)
	\end{aligned}\]
	has a bounded inverse. If $m=2$ then we should instead choose
	$\delta\in(3-2m,4-2m)$. In that case if we let $\chi$ be a
	compactly supported function on $Bl_0\mathbf{C}^m$ with non-zero
	integral, then the operator
	\begin{equation}\label{eq:op2}
		\begin{aligned}
		C^{4,\alpha}_\delta(Bl_0\mathbf{C}^m)^T\times\mathbf{R} &\to
		C^{0,\alpha}_{\delta-4}(Bl_0\mathbf{C}^m)^T \\
		(\phi, t) &\mapsto \mathcal{D}^*_\eta\mathcal{D}_\eta(\phi) +
		t\chi 
	\end{aligned}\end{equation}
	has bounded inverse.
\end{prop}
\begin{proof}
	This is also a consequence of duality theory in weighted spaces.
	Once again the image of 
	\[ \mathcal{D}^*\mathcal{D} :
	C^{4,\alpha}_\delta(Bl_0\mathbf{C}^m) \to
	C^{0,\alpha}_{\delta-4}(Bl_0\mathbf{C}^m)\]
	is the orthogonal complement of the kernel of
	\[ \mathcal{D}^*\mathcal{D} : C^{4,\alpha}_{4-2m-\delta}(Bl_0
	\mathbf{C}^m)
	\to C^{0,\alpha}_{-2m-\delta}(Bl_0\mathbf{C}^m).\]
	If $\delta > 4-2m$, then $4-2m-\delta < 0$. If $h\in
	\mathrm{Ker}\,\mathcal{D}^*\mathcal{D}$ and $h\in
	C^{4,\alpha}_\gamma(Bl_0\mathbf{C}^m)$ for some $\gamma < 0$
	then we must have $h=0$ (for the proof see~\cite{AP06}). This
	implies that our operator is surjective.

	When $m=2$ then the same argument shows that the image of
	$\mathcal{D}^*_\eta\mathcal{D}_\eta$ when $\delta\in(3-2m,4-2m)$
	has codimension 1, and more precisely the image is the
	subspace of functions with integral zero. It follows that the
	operator (\ref{eq:op2}) is surjective.
\end{proof}

\subsection{Weighted spaces on $Bl_pM$}
We will need to do analysis on the blown-up manifold $Bl_pM$ endowed
with the approximately extremal metric $\omega_\epsilon$ we constructed
in Section~\ref{sec:mainargument}. For this
we define the following weighted spaces, which are simply glued versions
of the above weighted spaces on $M\setminus\{p\}$ and
$Bl_p\mathbf{C}^m$.

We define the weighted H\"older norms $C^{k,\alpha}_\delta$ by
\[ \Vert f\Vert_{C^{k,\alpha}_\delta} = \Vert
f\Vert_{C^{k,\alpha}_{\omega}(M\setminus B_1)} + \sup_{\epsilon\leqslant
r\leqslant 1/2} r^{-\delta}\Vert
f\Vert_{C^{k,\alpha}_{r^{-2}\omega_\epsilon}(B_{2r}\setminus B_r)} +
\epsilon^{-\delta}\Vert f\Vert_{C^{k,\alpha}_{\eta}(B_\epsilon)}. \]

The subscripts indicate the metrics used to compute the relevant norm.
This is a glued version of the two spaces defined in the previous
section
in the following sense. If $f\in C^{k,\alpha}(Bl_pM)$ and we think of $Bl_pM$
as a gluing of $M\setminus\{p\}$ and $Bl_0\mathbf{C}^m$ then $\gamma_1f$
and $\gamma_2f$ can naturally be thought of as functions on
$M\setminus\{p\}$ and $Bl_0\mathbf{C}^m$ respectively. Then the norm
$\Vert f\Vert_{C^{k,\alpha}_\delta(Bl_pM)}$ is comparable to
\[ \Vert \gamma_1f\Vert_{C^{k,\alpha}_\delta(M_p)} + 
\epsilon^{-\delta}\Vert\gamma_2f\Vert_{C^{k,\alpha}_\delta
(Bl_0\mathbf{C}^m,\eta)}.\]

Another way to think about the norm is that if $\Vert
f\Vert_{C^{k,\alpha}_\delta}\leqslant c$ then $f$ is in
$C^{k,\alpha}(Bl_pM)$ and also for $i\leqslant k$ we have
\[ \begin{gathered}
	|\nabla^i f| \leqslant c\,\,\text{ for } r\geqslant 1\\
	|\nabla^i f| \leqslant cr^{\delta-i}\,\,\text{ for } \epsilon
	\leqslant r \leqslant 1\\
	|\nabla^i f| \leqslant c\epsilon^{\delta-i}
	\,\,\text{ for } r\leqslant\epsilon.
\end{gathered}\]
The norms here are computed with respect to the metric
$\omega_\epsilon$, and note that on $B_\epsilon$ we have
$\omega_\epsilon=\epsilon^2\eta$.

Sometimes we will restrict this norm to subsets such as
$C^{k,\alpha}_\delta(M\setminus B_{r_\epsilon})$ and
$C^{k,\alpha}_\delta(B_{2r_\epsilon})$. 
A crucial property of these weighted norms is that
\begin{equation}\label{eq:cutoff}
	\Vert \gamma_i\Vert_{C^{4,\alpha}_0} \leqslant c 
\end{equation}
for some constant $c$ independent of $\epsilon$.

In addition we need
the following lemma about lifting elements of
$\overline{\mathfrak{h}}\subset C^\infty(M)$ to 
$C^\infty(Bl_pM)$ according to Definition~\ref{def:lifting}.

\begin{lem}\label{lem:lb}
	For any $f\in\overline{\mathfrak{h}}$ its lifting satisfies
\[ \Vert \l(f) \Vert_{C^{0,\alpha}_0} \leqslant c|f|\]
and also $|X_{\l(f)}|_{\omega_\epsilon} \leqslant c|f|$
for some constant $c$ independent of $\epsilon$.
Here $|\cdot|$ is any fixed norm on $\overline{\mathfrak{h}}$.
\end{lem}
\begin{proof}
	Recall that we defined the lifting in
	Definition~\ref{def:lifting} using a decomposition
	$\overline{\mathfrak{h}}=\overline{\mathfrak{t}}
	\oplus\mathfrak{h}'$, where the
	functions in $\mathfrak{h}'$ vanish at $p$. Suppose
	first that $f\in\mathfrak{h}'$. Since $f$ vanishes at $p$, we
	have
	\[ \Vert f\Vert_{C^{1,\alpha}_1(M_p)}\leqslant c|f|,\]
	where $c$ is independent of $f$. It follows from the
	multiplication properties of weighted spaces and
	(\ref{eq:cutoff}) that
	\[ \Vert \l(f)\Vert_{C^{1,\alpha}_1} \leqslant c|f|,\]
	from which the required inequalities follow.

	Now suppose that $f\in\overline{\mathfrak{t}}$, 
	and write $X_f$ for the
	holomorphic vector field on $M$ corresponding to $f$. On the
	ball $B_{r_\epsilon}\subset M$,
	the action of $X_f$ is given by unitary transformations, and the
	size of the lifting to $B_{r_\epsilon}\subset Bl_pM$ is
	determined by the size of $X_f$ on $\partial B_{r_\epsilon}$.
	Outside $B_{r_\epsilon}$ the vector field is unchanged and the
	metrics $\omega$ and $\omega_\epsilon$ are uniformly equivalent.
	From these observations we can check that 
	$|X_{\l(f)}|_{\omega_\epsilon}\leqslant c|f|$ for some
	constant $c$. This in turn bounds $\nabla\l(f)$, from which
	the bound on $\Vert\l(f)\Vert_{C^{0,\alpha}_0}$ follows.
\end{proof}

\subsection{The linearized operator}
We now want to start studying the linearized operator (\ref{eq:linop}). 
The constants that appear below will be independent of
$\epsilon$ unless the dependence is made explicit. 

Recall that for any metric $\tilde\omega$ we defined
\[ L_{\tilde{\omega}}(\phi) = \Delta_{\tilde{\omega}}^2 \phi +
\mathrm{Ric}(\tilde{\omega})^{i\bar j} \phi_{i\bar j}.\]
We want to first study how this varies as we change the
metric.
For this we have

\begin{prop}\label{prop:linear}
	Suppose $\delta < 0$. 
	There exist constants $c_0, C>0$ such that if 
	$\Vert\phi\Vert_{C^{4,\alpha}_2} < c_0$
	then  	
	\[ \Vert L_{\omega_\phi}(f) -
	L_{\omega_\epsilon}(f)\Vert_{C^{0,\alpha}_{\delta-4}} \leqslant
	C\Vert\phi\Vert_{C^{4,\alpha}_2}\Vert
	f\Vert_{C^{4,\alpha}_\delta},\]
	where $\omega_\phi=\omega_\epsilon + \ddbar\phi$.
\end{prop}
\begin{proof}
	In the proof $c$ will denote a constant that may change from
	line to line, but is always independent of $\epsilon$.
	Let us write $g$, $g_\phi$ and for the Riemannian metrics
	corresponding to $\omega_\epsilon$ and $\omega_\phi$. 
	We can first choose $c_0$ small enough so that 
	$\frac{1}{2}g < g_\phi < 2g$, so the metrics are
	uniformly equivalent. We also have 
	\[\begin{aligned}
		\Vert g_\phi\Vert_{C^{2,\alpha}_0} & \leqslant c(1 +
		\Vert \phi\Vert_{C^{4,\alpha}_2}) \leqslant c,\\
	\end{aligned}
	\]
	where the norms are always computed with respect to
	$\omega_\epsilon$ and a set of coordinate charts obtained from a
	fixed set of charts on $M$ and on $Bl_0\mathbf{C}^m$.
	Schematically we have
	\[ \begin{gathered}
		\partial(g_{\phi}^{-1}) =
		g_\phi^{-2}\partial g_\phi \\ 
		\partial^2(g_\phi^{-1}) =
		g_\phi^{-3}\partial g_\phi\partial g_\phi +
		g_\phi^{-2}\partial^2
		g_\phi
	\end{gathered} \] 
	which implies, using our previous statements, that
	$\Vert g_\phi^{-1}\Vert_{C^{2,\alpha}_0}\leqslant c$.
	Since
	\[ g_\phi^{-1}-g^{-1} = g_\phi^{-1}(g - g_\phi)g^{-1}, \]
	we get 
	\[ \Vert g_\phi^{-1}-g^{-1}\Vert_{C^{2,\alpha}_0} \leqslant
	c \Vert\phi\Vert_{C^{4,\alpha}_2} .\]
	From this we can control
	$\Delta^2_{g_\phi}- \Delta^2_g$, since schematically
	\[ \begin{gathered}
		\Delta^2_{g_\phi}f - \Delta^2_g f =
		g_\phi^{-1}\partial^2(g_{\phi}^{-1}\partial^2f) -
		g^{-1}\partial^2(g^{-1}\partial^2f) \\
		= (g_\phi^{-1} - g^{-1})\partial^2(g_{\phi}^{-1}
		\partial^2f) + g^{-1}\partial^2\big[ (g_\phi^{-1} -
		g^{-1})\partial^2f\big],
	\end{gathered}\]
	from which we get
	\[ \begin{aligned}
		\Vert \Delta^2_{g_\phi}f - \Delta^2_{g}f\Vert_{
		C^{0,\alpha}_{\delta-4}}  \leqslant &\, \Vert g_\phi^{-1} -
		g^{-1}\Vert_{C^{0,\alpha}_0} \Vert 
		g_{\phi}^{-1}\Vert_{C^{2,\alpha}_0} \Vert\partial^2
		f\Vert_{C^{2,\alpha}_{\delta-2}} \\
		& + \Vert g^{-1}\Vert_{C^{0,\alpha}_0} \Vert
		g_\phi^{-1} - g^{-1}\Vert_{C^{2,\alpha}_0} \Vert
		\partial^2
		f\Vert_{C^{2,\alpha}_{\delta-2}} \\
		\leqslant &\,
		c\Vert\phi\Vert_{
		C^{4,\alpha}_2} \Vert
		f\Vert_{C^{4,\alpha}_\delta}.
	\end{aligned}\]
	For the terms involving the
	curvature, we first note that $\Vert
	\mathrm{Riem}(g)\Vert_{C^{k,\alpha}_{-2}}\leqslant c$ 
	for some constant
	independent of $\epsilon$. In addition 
	\[
	\Vert\mathrm{Riem}(g_\phi)-\mathrm{Riem}(g)\Vert_{C^{0,\alpha}_{-2}}
	\leqslant
	c\Vert\phi\Vert_{C^{4,\alpha}_2}.\]
	This follows from the schematic
	\[ \mathrm{Riem}(g) = \partial\Gamma + \Gamma\star\Gamma,\]
	where $\Gamma = g^{-1}\partial g$. In addition
	\[ \Gamma_\phi - \Gamma = (g_\phi^{-1} - g^{-1})\partial g_\phi
	+ g^{-1}(\partial g_\phi - \partial g),\]
	which implies that
	\[ \Vert \Gamma_\phi - \Gamma\Vert_{C^{1,\alpha}_{-1}} \leqslant
	c\Vert\phi\Vert_{C^{4,\alpha}_2},\]
	and so with similar calculations we get the required
	result for the curvature. 
\end{proof}

This result will have several useful consequences. 
First it allows us to estimate the nonlinear operator
$Q_{\omega_\epsilon}$ in the formula
\begin{equation}\label{eq:Q}
	\mathbf{s}(\omega_\epsilon + \ddbar\phi) =
\mathbf{s}(\omega_\epsilon) + L_{\omega_\epsilon}(\phi) +
Q_{\omega_\epsilon}(\phi).
\end{equation}

\begin{lem}\label{lem:contract}
Suppose that $\delta < 0$. There exist $c_0,C >
0$ such that if 
\[ \Vert \phi\Vert_{C^{4,\alpha}_2}, \Vert \psi
\Vert_{C^{4,\alpha}_2} \leqslant c_0, \]
then 
\[ \Vert Q_{\omega_\epsilon}(\phi) -
Q_{\omega_\epsilon}(\psi)\Vert_{C^{0,\alpha}_{\delta-4}} \leqslant
C\big\{\Vert\phi\Vert_{C^{4,\alpha}_2} + \Vert\psi\Vert_{C^{
4,\alpha}_2}\big\}\Vert \phi-\psi\Vert_{C^{4,\alpha}_\delta}. \]
\end{lem}
\begin{proof}
By the mean value theorem there exists some $\chi$, which is a convex
combination of $\phi$ and $\psi$, such that
\[ Q_{\omega_\epsilon}(\phi) - Q_{\omega_\epsilon}(\psi) =
DQ_{\omega_\epsilon,\chi}( \phi-\psi).\]
From Equation~\ref{eq:Q} we see that
$DQ_{\omega_\epsilon, \chi} =
L_{\omega_{\chi}}-L_{\omega_\epsilon}$,
so if $c_0$ is sufficiently small, then
from the previous proposition we get
\[ \Vert
DQ_{\omega_\epsilon,\chi}(\phi-\psi)\Vert_{C^{0,\alpha}_{\delta -4}}
\leqslant C\Vert\chi\Vert_{C^{4,\alpha}_2}
\Vert\phi-\psi\Vert_{C^{4,\alpha}_\delta}.\]
But $\Vert\chi\Vert_{C^{4,\alpha}_2}\leqslant
\Vert\phi\Vert_{C^{4,\alpha}_2} + \Vert\psi\Vert_{C^{4,\alpha}_2}$ 
so the required inequality holds.
\end{proof}

Next we want to study the invertibility of the linearized operator
(\ref{eq:linop}) of
our problem on $Bl_pM$. Let us write $X=\nabla\l(s)$, where
$\l(s)$ is the lift to $Bl_pM$ of the scalar curvature
$\mathbf{s}(\omega)$. 

\begin{prop}\label{prop:inv}
	For sufficiently small $\epsilon$ and $\delta\in(4-2m,0)$
	the operator 
	\[ \begin{gathered}
		G : (C^{4,\alpha}_\delta)^T \times
		\overline{\mathfrak{h}} \to
		(C^{0,\alpha}_{\delta-4})^T \\ 
		(\phi, f) \mapsto L_\omega(\phi) - \frac{1}{2}X(\phi)
		- \l(f)
	\end{gathered}\]
	has a right inverse $P$, with bounded operator norm $\Vert P\Vert <
	C$ for some constant $C$ independent of $\epsilon$.

	When $m=2$ and we choose $\delta=4-2m-\theta$ for
	$\theta > 0$ small,
	then we obtain a right inverse $P$ with $\Vert
	P\Vert < C\epsilon^{-\theta}$. 
\end{prop}
\begin{proof}
	This follows a standard argument for gluing solutions of linear
	problems, by first constructing an approximate inverse. See for
	example Chapter 7 in Donaldson-Kronheimer~\cite{DK90}.

	We will use the cutoff functions $\gamma_1,\gamma_2$ from
	before, where $\gamma_1+\gamma_2=1$, the function $\gamma_1$ is
	supported on $M\setminus B_{r_\epsilon}$, $\nabla\gamma_1$ is
	supported on $B_{2r_\epsilon}\setminus B_{r_\epsilon}$ and 
	\[ \Vert \gamma_i\Vert_{C^{4,\alpha}_0}\leqslant c.\]
	We will also need a cutoff
	function $\beta_1$ which is equal to 1 on the support of
	$\gamma_1$, such that $\nabla\beta_1$ is supported on a set
	slightly smaller than
	$B_{r_\epsilon}\setminus B_\epsilon$ and 
	$\beta_1=0$ on $B_\epsilon$. We will
	later choose $a < 1$ such that
	$r_\epsilon=\epsilon^a$, and for now let $\overline{a}$
	be such that $a < \overline{a} < 1$. Then we can define 
	\[ \beta_1(z) = \beta\left(\frac{\log |z|}{\log\epsilon}\right),
	\]
	where $\beta : \mathbf{R}\to\mathbf{R}$ is a fixed cutoff
	function such that $\beta(r)=1$ for $r<a$ and $\beta(r)=0$
	for $r > \overline{a}$.  
	The key point is that with this definition
	\[ \Vert \nabla\beta_1\Vert_{C^{3,\alpha}_{-1}} \leqslant
	\frac{c}{|\log\epsilon|},\]
	and also the support of $\nabla\beta_1$ is in $B_{r_\epsilon}
	\setminus B_{\epsilon^{\overline{a}}}$. 

	Similarly we define $\beta_2$ so that $\beta_2=1$ on the support
	of $\gamma_2$, but we want the support of $\nabla\beta_2$ to be
	slightly smaller than $B_1\setminus B_{2r_\epsilon}$. Namely
	we want $\nabla\beta_2$ to be supported on
	$B_{2\epsilon^{\underline{a}}}
	\setminus B_{2r_\epsilon}$, where $0<\underline{a}<a$. 
	Again we can define
	\[ \beta_2(z) = \tilde{\beta}\left(\frac{\log|z|/2}{\log
	r_\epsilon}\right),\]
	where $\tilde{\beta}:\mathbf{R}\to\mathbf{R}$ is a cutoff
	function such that $\beta(r)=0$ for $r < \underline{a}/a$ 
	and $\beta(r)=1$
	for $r > 1$. Once again, we obtain
	\[ \Vert \nabla \beta_2\Vert_{C^{3,\alpha}_{-1}}\leqslant
	\frac{c}{|\log r_\epsilon|}\leqslant \frac{c'}{|\log
	\epsilon|},\]
	for sufficiently small $\epsilon$.

	Let $\phi\in (C^{0,\alpha}_{\delta-4})^T$.
	The function $\gamma_1\phi$ can be 
	thought of as being defined on $M_p$. Since
	$\Vert\gamma_1\Vert_{C^{0,\alpha}_0}\leqslant c$ and the metrics 
	$\omega_\epsilon$ and $\omega$ are uniformly equivalent, we have 
	\[ \Vert \gamma_1\phi\Vert_{C^{0,\alpha}_{\delta-4}(M_p)} \leqslant
	c\Vert \phi\Vert_{C^{0,\alpha}_{\delta-4}}.\]
	It follows from Proposition~\ref{prop:Mp}
	that there exists some $f\in \overline{\mathfrak{h}}$ and
	$P_1(\gamma_1\phi)$ with
	\begin{equation}\label{eq:b1}
		\Vert
		P_1(\gamma_1\phi)\Vert_{C^{4,\alpha}_{\delta}(M_p)} +|f| 
		\leqslant
	        c\Vert\phi\Vert_{C^{0,\alpha}_{\delta -4}}
	\end{equation}
	for which
	\begin{equation}\label{eq:Mp}
		L_{\omega}P_1(\gamma_1\phi)
		-\frac{1}{2}\nabla\mathbf{s}(\omega)\cdot\nabla
		P_1(\gamma_1\phi)- f = \gamma_1\phi. 
	\end{equation}

	Similarly the 
	function $\gamma_2\phi$ can be thought of as a function
	on $Bl_0\mathbf{C}^m$, and from the definition of our norm we
	have
	\[ \Vert
	\gamma_2\phi\Vert_{C^{0,\alpha}_{\delta-4}(Bl_0\mathbf{C}^m)}
	\leqslant c\epsilon^{\delta-4}\Vert\phi\Vert_{C^{0,\alpha}_{
	\delta-4}}.\]
	From Proposition~\ref{prop:BlC}
	we have some $P_2(\gamma_2\phi)$
        with 
	\begin{equation}\label{eq:b2}
		\Vert
        	P_2(\gamma_2\phi)\Vert_{ 
		C^{4,\alpha}_{\delta}(Bl_0\mathbf{C}^m,\eta)}\leqslant
        	c\Vert\epsilon^4\gamma_2\phi\Vert_{
		C^{0,\alpha}_{\delta-4}(Bl_0\mathbf{C}^m,\eta)} \leqslant
		c\epsilon^\delta\Vert\phi\Vert_{C^{0,\alpha}_{\delta-4}},
	\end{equation}
        for which
	\[ L_{\eta}P_2(\gamma_2\phi) = 
		\epsilon^4\gamma_2\phi,\]
        so we also have
        \[ L_{\epsilon^2\eta}(P_2(\gamma_2\phi)) = \gamma_2\phi.\]
	We can think of $\beta_2P_2(\gamma_2\phi)$ as a function on
	$Bl_pM$ or on $Bl_0\mathbf{C}^m$. 
	
	We then define
	\[ P(\phi) = \beta_1P_1(\gamma_1\phi) +
	\beta_2P_2(\gamma_2\phi),\]
	where we are thinking of the annulus $B_1\setminus B_\epsilon$
	as a subset of $M_p$, $Bl_0\mathbf{C}^m$ and $Bl_pM$ at the same
	time. The bounds (\ref{eq:b1}) and (\ref{eq:b2}) imply that
	\begin{equation} \label{eq:b3}
		\Vert P(\phi)\Vert_{C^{4,\alpha}_\delta} \leqslant
		c\Vert\phi\Vert_{C^{0,\alpha}_{\delta-4}}.
	\end{equation}
	We want to show that the
	operator $\phi\mapsto (P(\phi),f)$ gives
	an approximate inverse to the operator $G$.
	\begin{claim}
		For sufficiently small $\epsilon$ we have
		\[ \Big\Vert L_{\omega_\epsilon}
		(P\phi) - \frac{1}{2}X(P\phi) -
		\l(f)  -
		\phi\Big\Vert_{C^{0,\alpha}_{\delta-4}} \leqslant
		\frac{1}{2} \Vert\phi\Vert_{C^{0,\alpha}_{\delta-4}}. \]
	\end{claim}
	To prove this note that we can write the expression we want to
	estimate as
	\[ \begin{gathered}
		L_{\omega_\epsilon}(\beta_1P_1(\gamma_1\phi))) -
		\frac{1}{2}X(\beta_1P_1(\gamma_1\phi)) -
		\gamma_1\l(f) - \gamma_1\phi \\
		+ L_{\omega_\epsilon}(\beta_2P_2(\gamma_2\phi)) -
		\frac{1}{2}X(\beta_2P_2(\gamma_2\phi)) -
		\gamma_2\l(f) - \gamma_2\phi,
	\end{gathered}\]
	where the terms on the top row are all supported in $M\setminus
	B_{\epsilon^{\overline{a}}}$ and the terms on the bottom
	row are supported in $B_{2\epsilon^{\underline{a}}}$.

	We first deal with $M\setminus
	B_{\epsilon^{\overline{a}}}$. 
	On this set we have
	$\omega_\epsilon=\omega+\ddbar\rho$, where
	\[ \rho(z) = \gamma_2(|z|)(-\phi(z) + 
		\epsilon^2\psi(\epsilon^{-1}z)).\]
	It follows that on the complement of $B_{\epsilon^{\overline{a}}}$ 
	we have
	\[ \Vert \rho\Vert_{C^{4,\alpha}_2(M\setminus B_{\epsilon^{
	\overline{a}}})}
	\leqslant c(r_\epsilon^2 + \epsilon^{(2m-2)(1-
	\overline{a})})=o(1),\]
	where by $o(1)$ we mean a constant going to zero as $\epsilon\to
	0$. 
	By the argument in Proposition~\ref{prop:linear}
	this implies that on the
	complement of $B_{\epsilon^{\overline{a}}}$ we have
	\[ \Vert L_{\omega_\epsilon}-L_{\omega}\Vert = o(1).\]
	At the same time $\mathbf{s}(\omega)-\l(s)$ is
	supported on $B_{2r_\epsilon}$ and is
	bounded in $C^{1,\alpha}_0$ by Lemma~\ref{lem:lb}. It follows
	from this that
	\begin{equation}\label{eq:Xgrads}
		\Vert (X - \nabla\mathbf{s}(\omega))\phi
		\Vert_{C^{0,\alpha}_{\delta-4}} \leqslant
		cr_\epsilon^2\Vert
		\phi\Vert_{C^{4,\alpha}_\delta}.
	\end{equation}
	Similarly, inside $B_{2\epsilon^{\underline{a}}}$ we have
	\[ \Big\Vert L_{\omega} + \frac{1}{2}X - L_{\epsilon^2\eta}
	\Big\Vert =
	o(1).\]
	Therefore it remains to show that for sufficiently small
	$\epsilon$ we have
	\begin{equation}\label{ineq:1}
		\Vert L_{\omega}(\beta_1P_1(\gamma_1\phi)) - \frac{1}{2}
		\nabla\mathbf{s}(\omega)\cdot\nabla
		(\beta_1P_1(\gamma_1\phi)) - \gamma_1\l(f) - \gamma_1\phi
		\Vert_{C^{0,\alpha}_{\delta-4}(M\setminus B_\epsilon)} <
		\frac{1}{4}\Vert\phi\Vert_{C^{0,\alpha}_{\delta-4}},
	\end{equation}
	and
	\[ \Vert L_{\epsilon^2\eta}(\beta_2P_2(\gamma_2\phi)) -
	\gamma_2\l(f)  -
	\gamma_2\phi\Vert_{C^{0,\alpha}_{\delta-4}(B_1)} < \frac{1}{4}
	\Vert\phi\Vert_{C^{0,\alpha}_{\delta-4}}.\]
	For the first inequality note, using Equation~(\ref{eq:Mp})
	that 
        \[ \begin{aligned}
              &L_{\omega}(\beta_1P_1(\gamma_1\phi)) - \frac{1}{2}
	      \nabla\mathbf{s}(\omega)\cdot\nabla
	      (\beta_1P_1(\gamma_1\phi)) - \gamma_1\l(f) -
              \gamma_1\phi \\
              = &\beta_1\gamma_1\phi + \beta_1f - \gamma_1\l(f) -
	       \gamma_1\phi + D^3(\nabla\beta_1\star
              P_1(\gamma_1\phi)) \\
	      = &\beta_1f - \gamma_1\l(f) + D^3(\nabla\beta_1\star
              P_1(\gamma_1\phi)), 
              \end{aligned}\]
          where $D^3$ denotes a 3rd order differential operator with
	  the coefficient of $\nabla^i$ bounded in $C^{4,\alpha}_{i-3}(M_p)$ 
	  and $\star$ is a bilinear algebraic operator. Since
          $\beta_1f - \gamma_1\l(f)$ is supported in
          $B_{2r_\epsilon}$ and is bounded in $C^{0,\alpha}_0$ by
          $c|f|$, we get
          \[ \Vert \beta_1f -
          \gamma_1\l(f)\Vert_{C^{0,\alpha}_{\delta-4}} \leqslant
          cr_\epsilon^{4-\delta}\Vert
          \phi\Vert_{C^{4,\alpha}_\delta}.\]
          Finally we have
          \[ \Vert D^3(\nabla\beta_1\star
          P_1(\gamma_1\phi))\Vert_{C^{0,\alpha}_{ \delta-4}} \leqslant
          c\Vert \nabla\beta_1\Vert_{C^{3,\alpha}_{-1}} \Vert
          P_1(\gamma_1 \phi)\Vert_{C^{3,\alpha}_{\delta}} = o(1)
          \Vert\phi\Vert_{C^{0,\alpha}_{\delta-4}},\]
	  so for small enough $\epsilon$ the Inequality~\ref{ineq:1}
	  holds.

          The proof of the second inequality is similar, just note that
          $\gamma_2\l(f)$ is supported in $B_{2r_\epsilon}$ so
          \[
          \Vert\gamma_2\l(f)\Vert_{C^{0,\alpha}_{\delta-4}(B_1)}
          \leqslant
          c r_\epsilon^{4-\delta}\Vert\phi\Vert_{C^{4,\alpha}_\delta}.
          \]

	  This proves the claim, so if we write 
	  $\tilde{P}(\phi) = (P\phi,f)$, this means that the
	  operator norm $\Vert G\circ\tilde{P}-I\Vert 
	  \leqslant \frac{1}{2}$,
	  which implies that we have a uniformly bounded inverse
	  $(G\circ\tilde{P})^{-1}$. This in turn shows that $G$ has a
	  right inverse $\tilde{P}\circ(G\circ\tilde{P})^{-1}$ whose
	  norm is bounded independently of $\epsilon$, 
	  which is what we wanted. 

	  When $m=2$ then we first work with the operator
	  \[\begin{gathered}
		  G_0 :
		  (C^{4,\alpha}_\delta)^T\times\overline{\mathfrak{h}}_0 
		  \times\mathbf{R}\to (C^{0,\alpha}_{\delta-4})^T \\
		  (\phi,f,t)\mapsto
		  \mathcal{D}^*_\omega\mathcal{D}_\omega(\phi) -
		  \l(f) + t\chi,
	  \end{gathered}\]
	  where $\chi$ is the function from Proposition~\ref{prop:BlC}. 
	  We 
	  are thinking of $\chi$ as a function on $Bl_pM$ using
	  the identification of $B_\epsilon\subset Bl_pM$ with
	  $B_1\subset Bl_0\mathbf{C}^m$ and in addition $\overline{
	  \mathfrak{h}}_0$
	  denotes the functions $f\in\overline{\mathfrak{h}}$ 
	  for which the lifting
	  $\l(f)$
	  has zero mean. The same argument as above can be used to show
	  that if $\epsilon$ is small then 
	  $G_0$ is invertible, with the inverse bounded
	  independently of $\epsilon$. Since
	  $\mathcal{D}^*\mathcal{D}(\phi)$ and $\l(f)$ have zero
	  mean, this implies that the operator
	  \[\begin{gathered}
		  G_1 :
		  (C^{4,\alpha}_\delta)^T\times\overline{\mathfrak{h}}_0
		  \to (C^{0,\alpha}_{\delta-4})^T_0 \\
		  (\phi,f)\mapsto
		  \mathcal{D}^*_\omega\mathcal{D}_\omega(\phi) -
		  \l(f)
	  \end{gathered}\]
	  also has bounded inverse, where
	  $(C^{0,\alpha}_{\delta-4})^T_0$ consist of functions with zero
	  mean. Now note that if $\delta < 4-2m$ then
	  \[ \Vert \phi\Vert_{L^1} \leqslant
	  c\epsilon^{\delta-(4-2m)}\Vert
	  \phi\Vert_{C^{0,\alpha}_{\delta-4}}.\]
	  It follows, by applying $G_1^{-1}$ to $\phi-\overline{\phi}$
	  and then absorbing the mean value $\overline{\phi}$ into $f$, that
	  the operator 
       	  \[\begin{gathered}
		  G_2 :
		  (C^{4,\alpha}_\delta)^T\times\overline{\mathfrak{h}} 
		  \to (C^{0,\alpha}_{\delta-4})^T \\
		  (\phi,f)\mapsto
		  \mathcal{D}^*_\omega\mathcal{D}_\omega(\phi) -
		  \l(f)
	  \end{gathered}\]
	  also has bounded inverse, but we only get $\Vert G_2^{-1}\Vert
	  < C\epsilon^{\delta-(4-2m)}=C\epsilon^{-\theta}$. Finally we
	  can compare $G_2$ with the operator $G$ in the statement of
	  the proposition using (\ref{eq:Xgrads}), and we find that $G$
	  is also invertible with the same bound, if $\theta$ is small
	  enough so that $r_\epsilon^2\ll \epsilon^\theta$.
\end{proof}

\subsection{The nonlinear equation}
We are now ready to solve the equation of Proposition~\ref{prop:gluing}
(see also Equation (\ref{eq:2})),
ie. we want to find a $T$-invariant function $u\in C^\infty(Bl_pM)$
and $f\in\overline{\mathfrak{h}}$ satisfying
\[ \mathbf{s}(\omega_\epsilon + \ddbar u) -
\frac{1}{2}\nabla\l(f + s)
\cdot\nabla u = \l(f+s).\]
Recall that here $\l(f)$ and $\l(s)$ are our lifts of the
functions $f,\mathbf{s}(\omega)\in \overline{\mathfrak{h}}$ 
to the blowup $Bl_pM$,
defined in Definition~\ref{def:lifting}.
As before we will write $X=\nabla\l(s)$. 
Let us write the equation as
\begin{equation}\label{eq:nonlinear}
 L_{\omega_\epsilon}(u) - \frac{1}{2}X(u) - \l(f) =
 \l(s) - \mathbf{s}(\omega_\epsilon) - Q_{\omega_\epsilon}(u) +
\frac{1}{2}\nabla\l(f)\cdot\nabla u.
\end{equation}

Following \cite{APS06} we first modify
$\omega$ on $M\setminus\{p\}$ so that it matches up with the 
Burns-Simanca metric to higher order. For this
let $\Gamma$ be a $T$-invariant solution of the linear equation
\begin{equation}\label{eq:Gamma}
	\mathcal{D}^*_\omega \mathcal{D}_\omega\Gamma = h\qquad
	\text{on } M\setminus\{p\}
\end{equation}
for some $h\in\overline{\mathfrak{h}}$, such that $\Gamma$ has an expansion
\[ \Gamma(z) = -|z|^{4-2m} + \tilde{\Gamma}, \]
where $\tilde{\Gamma} = O(|z|^{5-2m})$ for $m>2$ and $\Gamma$ has
leading term $\log|z|$ when $m=2$. It follows from this expansion
that $\Gamma$ is a distributional solution of
\[ \mathcal{D}^*_\omega\mathcal{D}_\omega \Gamma = h - c_m\delta_p,\]
where $c_m>0$ is a constant depending on the dimension and $\delta_p$ is
the delta function at $p$. We then find that for all
$g\in\overline{\mathfrak{h}}$
we have 
\[ \int_M gh\omega^m = c_m g(p),\]
so $h=c_m\mu(p)+\lambda$, where $\lambda=\mathrm{Vol}(M)^{-1}c_m$ is a
constant. 

Define the metric $\tilde{\omega}=\omega+\ddbar\Gamma$ on 
$M\setminus\{p\}$, so 
\[ \tilde{\omega} = \ddbar\left(\frac{|z|^2}{2} + \epsilon^{2m-2}\Gamma(z) +
\phi(z)\right),\]
where recall that $\omega = \ddbar(|z|^2/2 + \phi(z))$ near $p$. 
We can then write 
\[ \tilde{\omega} = \ddbar\left( \frac{|z|^2}{2} - \epsilon^{2m-2}
|z|^{4-2m} + \epsilon^{2m-2}\tilde{\Gamma}(z) + \phi(z)\right)\]
when $m>2$. 
At the same recall that the Burns-Simanca metric $\eta$ has an
expansion
\[ \eta = \ddbar\left(\frac{|z|^2}{2} - |z|^{4-2m} +
  \tilde{\psi}(z)\right),\]
where $\tilde{\psi} = O(|z|^{3-2m})$ for large $z$. We define the metric
$\tilde{\omega}_\epsilon$ by gluing
$\tilde{\omega}$ to $\epsilon^2\eta$ as before, to get on the annulus
$B_{2r_\epsilon}\setminus B_{r_\epsilon}$
\begin{equation}\label{eq:annulus}
 \begin{gathered}
	\tilde{\omega_\epsilon} = \ddbar\left( \frac{|z|^2}{2} -
 \epsilon^{2m-2}|z|^{4-2m} +
 \gamma_1(|z|)\left[\epsilon^{2m-2}\tilde{\Gamma}(z) +
 \phi(z)\right]\right. \\
 + \gamma_2(|z|) \epsilon^2\tilde{\psi}(\epsilon^{-1}
 z)\Big).
\end{gathered}
\end{equation}
Moreover outside $B_{2r_\epsilon}$ we have
$\tilde{\omega_\epsilon}=\tilde{\omega}$ while 
inside $B_{r_\epsilon}$ we have
$\tilde{\omega_\epsilon}=\epsilon^2\eta$. 
Note that in terms of our previous approximate metric
$\omega_\epsilon$ 
\begin{equation} \label{eq:newmetric}
	\tilde{\omega}_\epsilon = \omega_\epsilon +
	\ddbar\left[\epsilon^{2m-2}
	\gamma_1(|z|) \Gamma(z)\right].
\end{equation}
If $m=2$ then we can glue $\tilde{\omega}$ to $\epsilon^2\eta$ in the
same way.

We want to find a solution to Equation~\ref{eq:nonlinear} as a
perturbation of $\tilde{\omega}_\epsilon$ so we write
\begin{equation}\label{eq:expand}
\begin{gathered}
  u = \epsilon^{2m-2}\gamma_1\Gamma + v \\
    f = \epsilon^{2m-2}h + g.
\end{gathered}
\end{equation}
Substituting this into Equation~\ref{eq:nonlinear} and rearranging we get
\[ \begin{aligned}
L_{\omega_\epsilon}(v) - \frac{1}{2}X(v)-
\l(g) =& \l(s) - \mathbf{s}(\omega_\epsilon) -
Q_{\omega_\epsilon}(u) + \frac{1}{2}\nabla\l(f)\cdot \nabla u
\\
& -L_{\omega_\epsilon}(\epsilon^{2m-2}\gamma_1\Gamma) + \frac{1}{2}
X(\epsilon^{2m-2} \gamma_1\Gamma) + \epsilon^{2m-2}\l(h).
\end{aligned}\]
We can write this as a fixed point problem
\[ (v, g) = \mathcal{N}(v,g) \]
where we use the inverse $P$ constructed in Propostion~\ref{prop:inv}
and
\[\begin{aligned}
	\mathcal{N}(v,g) = & P\bigg\{  \l(s) - \mathbf{s}
	(\omega_\epsilon) -
	Q_{\omega_\epsilon}(\epsilon^{2m-2}\gamma_1\Gamma + v)\\
	&\quad + \frac{1}{2}\nabla(\epsilon^{2m-2}\l(h)+\l(g))\cdot
	\nabla\left(\epsilon^{2m-2}\gamma_1\Gamma + v\right) \\
&\quad 
 -L_{\omega_\epsilon}(\epsilon^{2m-2}\gamma_1\Gamma) + \frac{1}{2}
 X(\epsilon^{2m-2} \gamma_1\Gamma) + \epsilon^{2m-2}\l(h)\bigg\}
\end{aligned}\]

We first show that
\[ \mathcal{N} : (C^{4,\alpha}_\delta)^T\times\overline{\mathfrak{h}} 
\to (C^{4,\alpha}_\delta)^T\times\overline{\mathfrak{h}} \]
is a contraction on a small ball.
\begin{lem}\label{lem:contract2}
	There exist constants $c_0, \epsilon_0 > 0$ such that for
	$\epsilon < \epsilon_0$ the operator $\mathcal{N}$ is a
	contraction on the set
	\[ \{(v,g)\quad :\quad \Vert v\Vert_{C^{4,\alpha}_2}, 
		|g| < c_0 \}\]
	with constant $1/2$.
\end{lem}
\begin{proof}
        Suppose $m>2$. 
	Since $P$ is bounded independently of $\epsilon$, we need to
        control
        \[ Q_{\omega_\epsilon}(u_1) -
        \frac{1}{2}\nabla\l(f_1)\cdot\nabla u_1 -
        Q_{\omega_\epsilon}(u_2) +
        \frac{1}{2}\nabla\l(f_2)\cdot\nabla u_2,\]
	where
	\[ \begin{gathered}
		u_i = \epsilon^{2m-2}\gamma_1\Gamma + v_i \\
		f_i = \epsilon^{2m-2}h + g_i.
	\end{gathered}\]
	First note that
	\[
		\Vert\epsilon^{2m-2}\Gamma\Vert_{C^{k,\alpha}_2}
		\leqslant c(\epsilon r_\epsilon^{-1})^{2m-2} = o(1)
	\]
	as $\epsilon\to 0$. Hence for any $\lambda > 0$ we can choose
	sufficiently small $c_0$ and
	$\epsilon$ for which Lemma~\ref{lem:contract} implies that
        \[ \Vert Q_{\omega_\epsilon}(u_1) - Q_{\omega_\epsilon}(u_2)
        \Vert_{C^{0,\alpha}_{\delta -4}} \leqslant \lambda \Vert
        u_1-u_2\Vert_{C^{4,\alpha}_\delta} = \lambda \Vert v_1-v_2
	\Vert_{C^{4,\alpha}_\delta}.\]
        On the other hand we have
        \[ \begin{aligned}
          \Vert\nabla \l(f_1)\cdot \nabla u_1 &- \nabla \l(f_2)
	  \cdot\nabla u_2\Vert_{C^{0,\alpha}_{\delta-4}} \leqslant \\
          &\leqslant \Vert\nabla \l(f_1)\cdot (\nabla u_1 -
          \nabla u_2)\Vert_{C^{0,\alpha}_{\delta-4}} + 
          \Vert(\nabla \l(f_1) -
          \nabla \l(f_2))\cdot\nabla u_2\Vert_{C^{0,\alpha}_{\delta-4}}
          \\
          &\leqslant
          \Vert\nabla \l(f_1)\Vert_{C^{0,\alpha}_{-3}}
	  \Vert u_1-u_2\Vert_{C^{4,\alpha}_{\delta}}
          + \Vert u_2\Vert_{C^{4,\alpha}_{\delta}}\Vert
          \nabla \l(f_1-f_2)\Vert_{C^{0, \alpha}_{-3}} \\
          &\leqslant c(|f_1|\cdot\Vert u_1-u_2\Vert_{C^{4,\alpha}_\delta} +
          \Vert u_2\Vert_{C^{4,\alpha}_{\delta}} |f_1-f_2|),
          \end{aligned}\]
	  where we used Lemma~\ref{lem:lb}. 
          From this it is clear that by choosing $c_0$ and $\epsilon$
          sufficiently small, $\mathcal{N}$ is a contraction with
          constant $1/2$. 

	  If $m=2$ then $P$ is not bounded independently of $\epsilon$, 
	  but if we choose $\delta <
	  4-2m$ very close to $4-2m$ then the bound only blows up slowly
	  as $\epsilon\to0$ and the same argument works. 
\end{proof}

Next we need to bound $\mathcal{N}(0,0)$, which is the same as
estimating $\Vert F\Vert_{C^{0,\alpha}_{\delta-4}}$ where $F$ is the function
\begin{equation}\label{eq:F}
\begin{gathered}
	F= \l(s) - \mathbf{s}(\omega_\epsilon) -
Q_{\omega_\epsilon}(\epsilon^{2m-2}\gamma_1\Gamma) + 
\frac{1}{2}\nabla\epsilon^{2m-2}\l(h)\cdot \nabla\epsilon^{2m-2}
\gamma_1\Gamma
\\
 -L_{\omega_\epsilon}(\epsilon^{2m-2}\gamma_1\Gamma) + \frac{1}{2}
X(\epsilon^{2m-2} \gamma_1\Gamma) + \epsilon^{2m-2}\l(h).
\end{gathered}
\end{equation}  
\begin{lem}\label{lem:F}
	Choose $\delta$ very close to $4-2m$ with $\delta > 4-2m$ for
	$m>2$ and $\delta < 4-2m$ for $m=2$.  Let $r_\epsilon =
	\epsilon^{\frac{2m-1}{2m+1}}.$ Then 
	we have the estimate
	\[ \Vert F\Vert_{C^{0,\alpha}_{\delta-4}} \leqslant
	cr_\epsilon^{4-\delta},\]
	where $F$ is defined by Equation~(\ref{eq:F}).
\end{lem}
\begin{proof}
	To prove this we look at three
	different pieces of $Bl_pM$ separately, 
	namely $M\setminus B_{2r_\epsilon}$,
	$B_{2r_\epsilon}\setminus B_{r_\epsilon}$ and $B_{r_\epsilon}$. 
	First of all in $B_{r_\epsilon}$ we have $F = \l(s) +
	\epsilon^{2m-2}\l(h)$, but note that by Lemma~\ref{lem:lb}
	\[ \Vert \l(s) \Vert_{C^{0,\alpha}_0}, \quad \Vert
	  \l(h)\Vert_{C^{0,\alpha}_0 } \leqslant c,
	\]
	which implies that $\Vert F\Vert_{C^{0,\alpha}_{\delta-4}
	(B_{r_\epsilon})}
	\leqslant cr_\epsilon^{4-\delta}$. 

	On the set $M\setminus B_{2r_\epsilon}$ note that
	$\omega_\epsilon=\omega$, so $\l(s)=
	\mathbf{s}(\omega_\epsilon)$ and $\l(h) = h$. 
	In addition 
	\[ \mathcal{D}^*_\omega\mathcal{D}_\omega \Gamma =
	  L_\omega\Gamma - \frac{1}{2}X(\Gamma) = h
	\]
	using Equations~(\ref{eq:Gamma}) and (\ref{eq:Lichn}). This
	means that on the set $M\setminus B_{2r_\epsilon}$
	\begin{equation}\label{eq:F1}
		F = -Q_\omega(\epsilon^{2m-2}\Gamma) + \frac{1}{2}
	  \epsilon^{4m-4}\nabla h\cdot\nabla \Gamma. 
	\end{equation}
	It is useful to note that
	$\Vert\gamma_1\Gamma\Vert_{C^{4,\alpha}_w}$
	is bounded by $cr_\epsilon^{4-2m-w}$ for $w > 4-2m$ and by $c$
	for $w < 4-2m$. 
	For the second term in (\ref{eq:F1}) we have
	\[\begin{aligned}
		\epsilon^{4m-4}\Vert \nabla h\cdot \nabla\Gamma
		\Vert_{C^{0,\alpha}_{\delta-4}(M\setminus B_{2r_\epsilon})}
		&\leqslant c\epsilon^{4m-4} \Vert
		\nabla\Gamma\Vert_{C^{0,\alpha}_{\delta-4}} 
		\\
		&\leqslant c\epsilon^{4m-4} \Vert \Gamma\Vert_{
		C^{1,\alpha}_{\delta-3}} \\
		&\leqslant c\epsilon^{4m-4}
		\ll r_\epsilon^{4-\delta},
	\end{aligned}\]
	as long as $\delta$ is close to $4-2m$. 
	For the term involving $Q_\omega$ we use
	Proposition~\ref{prop:linear2} below. Indeed
	\[ \Vert Q_\omega(\epsilon^{2m-2}\Gamma)\Vert_{C^{0,\alpha}_{
	\delta-4}(M\setminus B_{2r_\epsilon})} \leqslant
	c\epsilon^{4m-4} r_\epsilon^{6-4m-\delta}\leqslant
	cr_\epsilon^{4-\delta},\]
	as long as $m\geqslant 2$. 

	Finally on the annulus $A_\epsilon = 
	B_{2r_\epsilon}\setminus B_{r_\epsilon}$ we
	first note that by Equation~(\ref{eq:newmetric}) we have
	\[ \mathbf{s}(\omega_\epsilon) +
	L_{\omega_\epsilon}(\epsilon^{2m-2} \gamma_1\Gamma) +
	Q_{\omega_\epsilon}(\epsilon^{2m-2}\gamma_1\Gamma) =
	\mathbf{s}(\tilde{\omega}_\epsilon).\]
	The other terms in the expression for $F$
	can be dealt with as before, so we only
	need to show that on the annulus $A_\epsilon$
	\[ \Vert \mathbf{s}(\tilde{\omega}_\epsilon)
	\Vert_{C^{0,\alpha}_{\delta-4}} \leqslant
	cr_\epsilon^{4-\delta}.\]
	This is where we use that $\tilde{\omega}$ matches up with the
	Burns-Simanca metric to leading order at $p$.
	We use the formula given in
	Equation~(\ref{eq:annulus}), which we write as
	\[ \tilde{\omega}_\epsilon = \ddbar\left( \frac{|z|^2}{2} +
	g(z)\right). \]
	On $A_\epsilon$ we then have
	\[ \mathbf{s}(\tilde{\omega}) = L_\eta(g),\]
	where $\eta = \ddbar\left(|z|^2/2 + tg(z)\right)$
	for some $t\in[0,1]$. We have
	\[ \mathbf{s}(\tilde{\omega}) = \Delta^2_0 g + (L_\eta -
	L_0)g,\]
	where $L_0 = \Delta^2_0$ is the linearized operator at the flat
	metric. At the same time
	\[ \Vert g \Vert_{C^{4,\alpha}_{\delta}(A_\epsilon)} 
		\leqslant c\epsilon^{2m-2}r_\epsilon^{4-2m-\delta}
		\ll r_\epsilon^{2-\delta},
	\]
	so Proposition~\ref{prop:linear2} implies that
	\[ \Vert L_\eta(g) - L_0(g)\Vert_{C^{0,\alpha}_{\delta
	-4}(A_\epsilon)} \leqslant
	cr_\epsilon^{\delta-2}(\epsilon^{2m-2}
	r_\epsilon^{4-2m-\delta})^2 \leqslant cr_\epsilon^{4-\delta}. \]
	
	Finally for $\Delta_0^2g$ note that $\Delta_0^2(|z|^{4-2m})=0$,
	so writing
	\[ g(z) = -\epsilon^{2m-2}|z|^{4-2m} + \tilde{g}(z)\]
	we have $\Delta_0^2g = \Delta_0^2\tilde{g}$. At the same time
	\[ \Vert \tilde{g} \Vert_{C^{4,\alpha}_\delta
	(A_\epsilon)}
	\leqslant c\epsilon^{2m-1}r_\epsilon^{3-2m-\delta} =
	cr_\epsilon^{4-\delta},\]
	so 
	\[ \Vert \Delta_0^2\tilde{g}\Vert_{C^{0,\alpha}_{
	\delta-4}(A_\epsilon)} \leqslant cr_\epsilon^{4-\delta},\]
	which gives the result we wanted.
\end{proof}

We used the following result, whose proof is identical to that of
Proposition~\ref{prop:linear}. 
\begin{prop}\label{prop:linear2}
	There exist constants $c_0, C>0$ such that if $U\subset Bl_pM$ and 
	$\Vert u\Vert_{C^{4,\alpha}_2}(U) \leqslant c_0$
	then for any $v$ we have	
	\[ \Vert L_{\omega_u}(v) -
	L_{\omega_\epsilon}(v)\Vert_{C^{0,\alpha}_{\delta-4}(U)} \leqslant
	C\Vert u\Vert_{C^{4,\alpha}_2(U)}\Vert
	v\Vert_{C^{4,\alpha}_\delta(U)},\]
	where $\omega_u=\omega_\epsilon + \ddbar u$. It follows
	that
	\[ \Vert Q_{\omega_\epsilon}(u)\Vert_{C^{0,\alpha}_{\delta-4}
	(U)} \leqslant C\Vert u\Vert_{C^{4,\alpha}_2(U)} \Vert
	u\Vert_{C^{4,\alpha}_\delta(U)}. \]
\end{prop}

We can now complete the proof of Proposition~\ref{prop:gluing}.  
Let us choose $\delta$ 
close to $4-2m$ with $\delta > 4-2m$ for $m>2$ and $\delta
< 4-2m$ for $m=2$, and let
$r_\epsilon=\epsilon^{\frac{2m-1}{2m+1}}$ as above.
First by Lemma~\ref{lem:F} and our bound on the inverse $P$ from
Proposition~\ref{prop:inv}, 
we have 
\[\Vert\mathcal{N}(0,0)\Vert_{C^{4,\alpha}_\delta} \leqslant
c_1r_\epsilon^{4-\delta}\epsilon^{-\theta}\] 
for some constant  $c_1$ 
independent of $\epsilon$, as
long as $\epsilon$ is sufficiently small. Here $\theta=0$ if $m>2$ and
$\theta=4-2m-\delta$ if $m=2$. Define
\[ S = \{ (v,g) \quad :\quad
\Vert v\Vert_{C^{4,\alpha}_\delta}, |g| \leqslant
2c_1r_\epsilon^{4-\delta}\epsilon^{-\theta}\}.\]
For $(v,g)\in S$ we have 
\[ \Vert v\Vert_{C^{4,\alpha}_2}\leqslant
Cr_\epsilon^{4-\delta}\epsilon^{\delta-2-\theta},\]
so if $\epsilon$ and $\theta$ are small enough then
Lemma~\ref{lem:contract2} implies that 
$\mathcal{N}$ is a contraction
with constant $1/2$ on $S$. In particular $\mathcal{N}$ then maps $S$ to
itself, since if $(v,g)\in S$ then 
\[ \Vert\mathcal{N}(v,g)\Vert \leqslant
\Vert\mathcal{N}(v,g) - \mathcal{N}(0,0)\Vert + \Vert\mathcal{N}(0,0)
\Vert \leqslant \frac{1}{2}\Vert(v,g)\Vert +
c_1r_\epsilon^{4-\delta}\epsilon^{-\theta}
\leqslant 2c_1 r_\epsilon^{4-\delta}\epsilon^{-\theta}.\]
It follows that for small enough $\epsilon$ there is a fixed point of
$\mathcal{N}$ in the set $S$. This gives a solution $(v,g)$ to our
equation, with $|g|\leqslant 2c_1r_\epsilon^{4-\delta}\epsilon^{-\theta}$. 
Finally if
$\delta$ is sufficiently close to $4-2m$, we find that
$r_\epsilon^{4-\delta} < \epsilon^\kappa$ for some $\kappa > 2m-2$.
Hence from the expansion (\ref{eq:expand}) and the fact that
$h=c_m\mu(p)$, we
obtain the required expansion in Equation (\ref{eq:expansion}). 

\subsection{A remark on Conjecture~\ref{conj:stronger}}
A natural problem is to compute more terms in the expansion
(\ref{eq:expand}) of the element $f$ above. Examining the argument we
see that the key point was to first perturb the extremal metric $\omega$
away from $p$, so that it matches up with the Burns-Simanca metric to
higher order. To see what the next term should be we need the following.

\begin{lem}\label{lem:BSmetric}
	If $m\geqslant 3$ then the K\"ahler potential for a suitable
	scaling of the
	Burns-Simanca metric 
	\[\eta=\ddbar(|z|^2/2 + \psi(z))\]
	satisfies
	\begin{equation}\label{eq:metric}
		\psi(z) = -|z|^{4-2m} + a|z|^{2-2m} + O(|z|^{6-4m}),
	\end{equation}
	where $a > 0$. 
\end{lem}
\begin{proof}
This can be seen by finding the first few terms in the power series
expansion of the solution of the ODE for scalar flat $U(n)$ invariant
metrics on $Bl_0\mathbf{C}^m$, written down in~\cite{AP06}, Section 7 
(see also~\cite{Sim91}). Following \cite{AP06}
let us write $\eta=\ddbar A(|z|^2)$ and let $s=|z|^2$. Let us also
introduce the variable $t=s^{-1}$ and define the function
$\xi(t)=\partial_s A(s)$. From the equations given in \cite{AP06} one
can check that $\xi$ satisfies the equation
\[ \xi^{m-1}(t)\xi'(t) - (m-1)t^{m-2}\xi(t) + (m-2)t^{m-1} = 0.\]
Moreover we want $\xi(0)=1/2$. It is then straightforward to check that
the first few terms in the expansion of $\xi$ around $t=0$ are
\[ \xi(t) = \frac{1}{2} + 2^{m-2}t^{m-1} - \frac{m-2}{m}2^{m-1}t^m +
O(t^{m+1}).\]
From this we can recover $A(s)$, and finally by scaling the variable $z$
and the metric, we obtain the first two terms in (\ref{eq:metric}), with
$a>0$. 

To show that the next term is $O(|z|^{6-4m})$ we can either compute more
terms in the expansion of $\xi$, or instead we can follow the argument
in \cite{AP06}, Lemma 7.2. The scalar curvature of $\eta=\ddbar(|z|^2/2
+ \psi(z))$ is given by
\[ \mathbf{s}(\eta) = \Delta^2\psi + Q(\psi),\]
where $\Delta$ is the Laplacian for the flat metric (we use the K\"ahler
Laplacian and half the Riemannian scalar curvature, so the coefficient
of $\Delta^2$ differs from that in \cite{AP06}). It is shown
in~\cite{AP06} that if
$\psi\in C^{4,\alpha}_\delta(Bl_0\mathbf{C}^m)$ for some $\delta < 2$
then 
\[ Q(\psi)\in C^{0,\alpha}_{2\delta - 6}(Bl_0\mathbf{C}^m).\]
If $\mathbf{s}(\eta)=0$ then
also $\Delta^2\psi\in C^{0,\alpha}_{2\delta-6}$, so from the
regularity theory for the Laplacian acting between weighted
spaces we get that 
\begin{equation}\label{eq:exp}
	\psi \in C^{4,\alpha}_{2\delta-2} \oplus \mathrm{span}\{1,
|z|^{4-2m}, |z|^{2-2m}\}.
\end{equation}
The reason why we only get these powers of $|z|$ is that these (together
with $|z|^2$) are the only $U(n)$ invariant elements in the kernel of
$\Delta^2$. Or in other words if we work with $U(n)$ invariant spaces
then the indicial roots are $2,0,4-2m$ and $2-2m$. Subtracting a constant we
can therefore suppose that $\psi\in C^{4,\alpha}_{4-2m}$ so we can apply
the above with $\delta=4-2m$. Then $2\delta-2=6-4m$ so the result
follows from (\ref{eq:exp}). 
\end{proof}

In order to match with the metric $\epsilon^2\eta$ we therefore need to
perturb $\omega$ to $\omega+\ddbar\Gamma$, where 
\[ \Gamma = -\epsilon^{2m-2}|z|^{4-2m} + \epsilon^{2m}a|z|^{2-2m} +
\text{lower order terms}\]
and $\mathcal{D}_\omega^*\mathcal{D}_\omega\Gamma = h$ on
$M\setminus\{p\}$ for some $h\in\overline{\mathfrak{h}}$. 
By changing the lower
order terms if necessary (we can use a term of the order of
$\epsilon^{2m}|z|^{4-2m}$ to cancel the contribution of
$\mathrm{Ric}^{i\bar j}\Gamma_{i\bar j}$) we can assume that $\Gamma$ is
a distributional solution of 
\[ \mathcal{D}_\omega^*\mathcal{D}_\omega\Gamma = h -
\epsilon^{2m-2}c_m\delta_p - \epsilon^{2m}ac_m'\Delta\delta_p,\]
where $c_m,c_m'>0$ are constants depending on the dimension. Taking the
$L^2$ product of both sides with all $g\in\overline{\mathfrak{h}}$ as before, 
we find that
\[ h = \epsilon^{2m-2}(\lambda+c_m\mu(p)) +
\epsilon^{2m}ac_m'\Delta\mu(p),\]
where $\lambda=\mathrm{Vol}(M)^{-1}c_m$. 
Under the assumptions of Conjecture~\ref{conj:stronger}, if $\epsilon$
is sufficiently small we can assume that $h=\epsilon^{2m-2}\lambda$. 
It seems reasonable to
expect that one can deform this metric $\omega+\ddbar\Gamma$ to a cscK
metric, but we have not been successful with this so far.  Note also
that when $m=2$, then the potential for the Burns-Simanca metric is
given by $\psi=\log|z|$ with no lower order terms, so it is not clear
where the expression $\mu(p) \pm \epsilon\Delta\mu(p)$, which we see in
the algebro-geometric calculations in the next section, comes from in this
case.

\section{K-stable blowups}\label{sec:stability}
In this section we give the proof of Theorem~\ref{thm:stable}, which is
an extension of a result in Stoppa~\cite{Sto10}. 
First we need to review the notion of K-stability
introduced by Donaldson~\cite{Don02}. 

Let $L\to M$ be an ample line bundle. A \emph{test-configuration} 
for the pair
$(M,L)$ is a flat $\mathbf{C}^*$-equivariant family $\pi:\mathcal{M}\to 
\mathbf{C}$ together with a $\mathbf{C}^*$-equivariant relatively ample
line bundle $\mathcal{L}\to\mathcal{M}$, such that the fiber
$(\pi^{-1}(1), \mathcal{L}|_{\pi^{-1}(1)})$ is isomorphic to $(M,L^r)$
for some $r > 0$. Let us denote by $\alpha$ the induced 
$\mathbf{C}^*$-action on the
central fiber $(M_0,L_0)$. This gives rise to a
$\mathbf{C}^*$-action on the space of sections $H^0(M_0,L_0^k)$ for each
$k$. Let us write $d_k = \dim H^0(M_0,L_0^k)$ and $w_k$ for the total
weight of the action on $H^0(M_0,L_0^k)$. Define the numbers
$a_0,a_1,b_0,b_1$ to be the coefficients in the expansions
\[ \begin{gathered}
	d_k = a_0k^m + a_1k^{m-1} + \ldots \\
	w_k = b_0k^{m+1} + b_1k^m + \ldots,
\end{gathered}\]
valid for large $k$. We define the \emph{Futaki invariant} of the action
$\mathbf{C}^*$-action $\alpha$ on $(M_0,L_0)$ to be 
\[ \mathrm{Fut}(\alpha, M_0, L_0) = \frac{a_1}{a_0}b_0 - b_1.\]
Since the $\mathbf{C}^*$-action is induced by the test-configuration
$(\mathcal{M},\mathcal{L})$,
we also write
\[ \mathrm{Fut}(\mathcal{M},\mathcal{L}) = \mathrm{Fut}(\alpha,M_0,L_0). \]  
\begin{defn}
	The polarized manifold $(M,L)$ is \emph{K-polystable} if for all
	test-con\-fi\-gu\-rations
	$\mathrm{Fut}(\mathcal{M},\mathcal{L})\geqslant 0$ with equality
	only if the central fiber of the test-configuration
	is isomorphic to $M$.
\end{defn}

We want to study the K-stability of the blowup $Bl_pM$ with the
polarization $\pi^*L-\epsilon E$ for sufficiently small $\epsilon$. For
this the key calculation is to compute the Futaki invariants of
$\mathbf{C}^*$-actions on $Bl_pM$. 

Let us fix a K\"ahler metric $\omega\in c_1(L)$, and suppose that the
vector field 
$v = J\nabla h$ generates a holomorphic $S^1$-action for some $h\in
C^\infty(M)$. 
Replacing $L$ by a large power if necessary, a choice of $h$ gives rise to a
lifting of the $S^1$-action to $L$ which can be extended to a
holomorphic $\mathbf{C}^*$-action. We define the Futaki invariant
of $v$ with the same formula as above. 
If the vector field $v$ vanishes at $p$ then it has a
holomorphic lift $\tilde{v}$ to $Bl_pM$. Consider the $\mathbf{Q}$-line
bundle $L_\epsilon = \pi^*L-\epsilon E$ on $Bl_pM$ for small rational
$\epsilon$, where $E$ is the
exceptional divisor. We have a
$\mathbf{C}^*$-action on the space of sections $H^0(Bl_pM,
L_\epsilon^k)$, and so we can define constants $\tilde{a}_0,\tilde{a}_1,
\tilde{b}_0, \tilde{b}_1$ corresponding to this action as above, which
depend on $\epsilon$ (if we take $k$ for which $k\epsilon$
is an integer, then $L_\epsilon^k$ is a line bundle). 
The following lemma is an extension of the calculation in 
Stoppa~\cite{Sto10}.

\begin{lem}\label{lem:weights}
	For the action on the blowup we have
	\[\begin{aligned}
		\tilde{a}_0 &= a_0 - \frac{\epsilon^m}{m!} \\
		\tilde{a}_1 &= a_1 - \frac{\epsilon^{m-1}}{2(m-2)!} \\
		\tilde{b}_0 &= b_0 + \frac{\epsilon^m}{m!}h(p) +
		\frac{\epsilon^{m+1}}{(m+1)!}\Delta h(p) \\
		\tilde{b}_1 &= b_1 + \frac{\epsilon^{m-1}}{2(m-2)!}h(p)
		+\frac{(m-2)\epsilon^m}{2m!}\Delta h(p),
	\end{aligned}\]
	where $\Delta h$ is the Laplacian of $h$ with respect to the
	metric $\omega$.
\end{lem}
\begin{proof}
	Let us write $\mathcal{I}_p$ for the ideal sheaf of $p\in M$.
	For large $k$ we have an isomorphism
	\[ H^0(Bl_pM, L_\epsilon^k) = H^0(M,
	\mathcal{I}_p^{k\epsilon}L^k).\]
	To study this space, we use the exact sequence 
	\[ 0\longrightarrow \mathcal{I}_p^{k\epsilon}L^k \longrightarrow
	L^k \longrightarrow \mathcal{O}_{k\epsilon p}\otimes L^k|_p
	\longrightarrow 0.\]
	As before, let us write $d_k$ and $w_k$ for the dimension of
	$H^0(M,L^k)$ and the weight of the action on this space.
	Similarly write $\tilde{d}_k$ and $\tilde{w}_k$ for the
	dimension of, and weight of the action on, $H^0(Bl_pM,
	L_\epsilon^k)$. From the exact sequence we have
	\begin{equation}\label{eq:dkwk}
		\begin{aligned}
		\tilde{d}_k &= d_k - \dim \mathcal{O}_{k\epsilon p} \\
		\tilde{w}_k &= w_k - w(\mathcal{O}_{k\epsilon p}\otimes
		L^k|_p).
	\end{aligned}\end{equation}
	Here the weight $w(\mathcal{O}_{k\epsilon p}\otimes L^k|_p)$ is
	given by
	\[ w(\mathcal{O}_{k\epsilon p}\otimes L^k|_p) =
	w(\mathcal{O}_{k\epsilon p}) - kh(p)\dim(\mathcal{O}_{
	k\epsilon p}),\]
	since the weight of the action on the fiber $L_p$ is 
	$-h(p)$. The sign here depends on our convention that the real
	part of the
	$\mathbf{C}^*$-action corresponding to $h$ is generated by
	$\nabla h$. 
	
	We can think of $\mathcal{O}_{lp}$ for an integer $l>0$ as being
	the space of $(l-1)$-jets of functions at $p$, ie.
	\[ \mathcal{O}_{lp} = \mathbf{C}\oplus T_p^*\oplus\ldots\oplus
	S^{l-1}T_p^*,\]
	where $S^i$ is the $i^\mathrm{th}$ symmetric product. The
	dimension of $\mathcal{O}_{lp}$ is therefore given by
	\[ \dim(\mathcal{O}_{lp}) = \binom{m+l-1}{m} =
	\frac{1}{m!}\left( l^m + \frac{m(m-1)}{2} l^{m-1} +
	O(l^{m-2})\right).\]
	Similarly if we write $w$ for the weight of the action on
	$T_p^*$, then we can compute that
	\[ w(\mathcal{O}_{lp}) = \binom{m+l-1}{m+1}w = \frac{w}{(m+1)!}
	\left( l^{m+1} + \frac{(m-2)(m+1)}{2}l^m + O(l^{m-1})\right).\]
	Substituting $k\epsilon$ for $l$ and using the formulas
	(\ref{eq:dkwk}) we get
	\[ \begin{aligned}
		\tilde{d}_k &= d_k - \frac{\epsilon^m}{m!} k^m -
		\frac{\epsilon^{m-1}}{2(m-2)!} k^{m-1} + O(k^{m-2}) \\
		\tilde{w}_k &= w_k -
		\left(\frac{w\epsilon^{m+1}}{(m+1)!} -
		\frac{h(p)\epsilon^m}{m!}\right)k^{m+1} \\
		&\quad -
		\left(\frac{(m-2)w\epsilon^m}{2m!} -
		\frac{h(p)\epsilon^{m-1}}{2(m-2)!}\right)k^m +
		O(k^{m-1})
	\end{aligned}\]
	The only thing that remains is to see that the weight of the
	action on $T_p^*$ is given by $w=-\Delta h(p)$. This follows from
	the fact that by our convention 
	the induced action on the tangent space $T_p$ is
	given by the Hessian of $h$ at $p$.
\end{proof}

A simple calculation then gives
\begin{cor}\label{cor:Fblowup}
	If the Futaki invariant $\mathrm{Fut}(v,M,L)=0$ on $M$
	then on the blowup we have
	\begin{equation}\label{eq:Fut} \begin{aligned}
		\mathrm{Fut}(\tilde{v}, &Bl_pM, L_\epsilon) = 
		\frac{\tilde{a}_1}{\tilde{a}_0}\tilde{b}_0-\tilde{b}_1
		\\
	  	&= -\frac{\epsilon^{m-1}}{2(m-2)!}h(p) -
		\frac{\epsilon^m}{m!}
		\left( \frac{m-2}{2}\Delta h(p) - \frac{a_1}{a_0}h(p)\right) +
		O(\epsilon^{m+1}),
	\end{aligned}\end{equation}
	if $m\geqslant 3$,
	\[ \begin{aligned}
		\mathrm{Fut}(\tilde{v}, Bl_pM, L_\epsilon) &=
		-\frac{\epsilon}{2}h(p) + \frac{\epsilon^2a_1}{2a_0}
		h(p)+\frac{\epsilon^3}{2a_0}\left(
		\frac{a_1}{3}\Delta h(p)- \frac{h(p)}{2}\right) + O(\epsilon^4),
	\end{aligned}\]
	if $m=2$ and $a_1\not=0$ and finally
	\[ \begin{aligned}
		\mathrm{Fut}(\tilde{v}, Bl_pM, L_\epsilon) &=
		-\frac{\epsilon}{2}h(p) +\frac{\epsilon^3}{4a_0}h(p) -
		\frac{\epsilon^4}{12a_0}\Delta h(p) + O(\epsilon^5),
	\end{aligned}\]
	if $m=2$ and $a_1=0$. 
	
	In each case if $h(p)=\Delta h(p)=0$, then
	$\mathrm{Fut}(\tilde{v}, Bl_pM, L_\epsilon)=0$ for all $\epsilon$. 
\end{cor}

Combining this with Proposition~\ref{prop:HM} from
Section~\ref{sec:deform} we can prove Theorem~\ref{thm:stable}.

\begin{proof}[Proof of Theorem~\ref{thm:stable}]
	Let us assume that $m>2$, since the argument in the $m=2$ case
	is essentially identical. 
	We argue by contradiction. We can also assume that $p$ is
	semistable with respect to the polarization $L$ since if it
	were strictly unstable, then Stoppa's result~\cite{Sto10} implies
	that the pair $(Bl_pM, L_\epsilon)$ is K-unstable for all
	sufficiently small $\epsilon$.
	
	So suppose that  $p$ is semistable with respect to $L$, but it
	is not polystable with respect to $L+\epsilon K_M$ 
	for sufficiently small $\epsilon$. Then by
	Proposition~\ref{prop:HM} there exists a one-parameter subgroup
	$\lambda$ such that $\lim\limits_{t\to 0}\lambda(t)\cdot p=q$
	and $w_L(q,\lambda)=0$, but $w_{K_M}(q,\lambda)\leqslant 0$. 
	By blowing up the trivial
	family $M\times\mathbf{C}$ in the closure of the
	$\mathbf{C}^*$-orbit 
	of $(p,1)$ under the action $t(p,1)=(\lambda(t)p,t)$ we 
	obtain a test-configuration
	for $(Bl_pM, L_\epsilon)$ with central fiber $(Bl_qM,
	L_\epsilon)$, and the action on $Bl_qM$ is simply the lifting of
	the action $\lambda$. Suppose that this $\lambda$ is
	generated by a holomorphic vector field with Hamiltonian
	function $h$. Then $w_L(q,\lambda)=-h(q)=0$  
	and $w_{K_M}(q,\lambda)=-\Delta h(q)\leqslant 0$. 
	If $-\Delta h(q)< 0$ then from the formula (\ref{eq:Fut}) we
	see that for small enough $\epsilon$ the Futaki invariant of
	this test-configuration is negative. Whereas if $\Delta h(q)=0$
	then we get that the Futaki invariant is zero. In both cases
	this means that $(Bl_pM,L_\epsilon)$ is not K-polystable.

	To deal with the different situations when $m=2$, note that by
	the Hirzebruch-Riemann-Roch formula 
	$a_1 = -\frac{1}{2}K_M\cdot L$.
\end{proof}

To relate GIT stability with respect to $L+\delta K_M$ (or
$L-\delta K_M$) to
moment maps, we need the following.
\begin{lem}\label{lem:LK}
	Let $X$ be a holomorphic Killing field on $(M,\omega)$. If
	$\iota_X\omega = dh$, then $\iota_X\rho = -d\Delta h$, where
	$\rho$ is the Ricci form. It follows that if $p$ is polystable
	with respect to the polarization $L+\delta K_M$ for the action
	of the Hamiltonian isometry group $H$, then there
	is a point $q$ in the complex orbit $H^c\cdot p$ such that
	$\mu(q)+\delta\Delta\mu(q)=0$. 
\end{lem}
\begin{proof}
	We can compute
	\[ \begin{aligned}
		2\iota_X\rho &= \iota_X(dJd\log\det(\omega)) 
		=-d\iota_X(Jd\log\det(\omega)) \\
		&=d(\mathcal{L}_{JX}\log\det(\omega)) 
		=d\Lambda \mathcal{L}_{JX}\omega,
	\end{aligned}\]
	where $\mathcal{L}$ is the Lie derivative and $\Lambda$ is
	taking the trace with respect to $\omega$. Since
	$\mathcal{L}_{JX}\omega = -2\ddbar h$, we get $\iota_X\rho =
	-d\Delta h$, which is what we wanted.

	It follows that if $\mu$ is a moment map for the action of $H$
	with respect to the symplectic form $\omega$ then
	$\mu+\delta\Delta\mu$ is a moment map with respect to
	$\omega-\delta\rho$. Since $\omega-\delta\rho\in c_1(L+\delta
	K_M)$, the second statement in the lemma follows from the
	Kempf-Ness theorem.
\end{proof}

Finally we give the proof of Corollary~\ref{cor:KE}.
\begin{proof}[Proof of Corollary~\ref{cor:KE}]
We will prove the following three implications.
	\begin{itemize}
		\item[(i)$\Rightarrow$(ii)] This follows from the main
			theorem in~\cite{SSz09}.
		\item[(ii)$\Rightarrow$(iii)] This follows from
			Theorem~\ref{thm:stable}, even though we are
			using a more restrictive version of
			K-polystability. The reason is that in finite
			dimensions when using the Hilbert-Mumford
			criterion for testing stability of a point, it
			is enough to look at one-parameter subgroups
			which commute with a torus fixing the point.
			This follows for example from the theory of
			optimal destabilizing one-parameter subgroups
			(see Kempf~\cite{Kem78}), or also from the
			Kempf-Ness theorem and the observation that
			$\mu(p)$ is always in the center of the
			stabilizer of $p$ since $\mu$ is equivariant. 
		\item[(iii)$\Rightarrow$(i)] This follows from
			Theorem~\ref{thm:main}. Namely if $p\in M$ is
			GIT polystable then by replacing $p$ by a
			different point in its $H^c$-orbit, we can
			assume that $\mu(p)=0$, so $Bl_pM$ admits
			an extremal metric in the class
			$\pi^*[\omega]-\epsilon[E]$ for small $\epsilon$. 
			Since $(M,\omega)$ is K\"ahler-Einstein, the
			Futaki invariant of any vector field on $M$ is
			zero with respect to the class $[\omega]$. Since
			$\mu(p)=0$ and $\rho=\omega$, Lemma~\ref{lem:LK}
			implies that $\Delta\mu(p)=0$, so from
			Corollary~\ref{cor:Fblowup} the Futaki invariant
			of every vector field on $Bl_pM$ vanishes in the
			classes $\pi^*[\omega]-\epsilon[E]$. This means
			that the extremal metric we obtain actually has
			constant scalar curvature.
	\end{itemize}
\end{proof}

\providecommand{\bysame}{\leavevmode\hbox to3em{\hrulefill}\thinspace}
\providecommand{\href}[2]{#2}

\bigskip
\noindent{\sc Columbia University\\ New York}
\\
\noindent{\tt gabor@math.columbia.edu}


\begin{thebibliography}{10}

\bibitem{AP07}
C.~Arezzo and F.~Pacard, \emph{On the {K}\"ahler classes of constant scalar
  curvature metrics on blow ups}, preprint.

\bibitem{AP06}
\bysame, \emph{Blowing up and desingularizing constant scalar curvature
  {K}\"ahler manifolds.}, Acta Math. \textbf{196} (2006), no.~2, 179--228.

\bibitem{AP09}
\bysame, \emph{Blowing up {K}\"ahler manifolds with constant scalar curvature
  {II}}, Ann. of Math. (2) \textbf{170} (2009), no.~2, 685--738.

\bibitem{APS06}
C.~Arezzo, F.~Pacard, and M.~A. Singer, \emph{Extremal metrics on blow ups},
  preprint.

\bibitem{BM48}
S.~Bochner and W.~T. Martin, \emph{Several complex variables}, Princeton
  Mathematical Series, vol.~10, Princeton University Press, Princeton, N.J.,
  1948.

\bibitem{Cal82}
E.~Calabi, \emph{Extremal {K}\"ahler metrics}, Seminar on Differential Geometry
  (S.~T. Yau, ed.), Princeton, 1982.

\bibitem{Cal85}
E.~Calabi, \emph{Extremal {K}\"ahler metrics II}, Differential geometric
and complex analysis, Springer, 1985.
 
\bibitem{DV08}
A.~Della~Vedova, \emph{{CM}-stability of blow-ups and canonical metrics},
  arXiv:0810.5584.

\bibitem{Don02}
S.~K. Donaldson, \emph{Scalar curvature and stability of toric varieties}, J.
  Differential Geom. \textbf{62} (2002), 289--349.

\bibitem{Don05}
\bysame, \emph{Lower bounds on the {C}alabi functional}, J. Differential Geom.
  \textbf{70} (2005), no.~3, 453--472.

\bibitem{Don08_1}
\bysame, \emph{Constant scalar curvature metrics on toric surfaces}, Geom.
  Funct. Anal. \textbf{19} (2009), no.~1, 83--136.

\bibitem{DK90}
S.~K. Donaldson and P.~B. Kronheimer, \emph{The geometry of four-manifolds},
  OUP, 1990.

\bibitem{Hong02}
Y.-J. Hong, \emph{Gauge-fixing constant scalar curvature equations on ruled
  manifolds and the {F}utaki invariants}, J. Differential Geom. \textbf{60}
  (2002), no.~3, 389--453.

\bibitem{Hong08}
Y.~J. Hong, \emph{Stability and existence of critical {K}\"ahler metrics on
  ruled manifolds}, J. Math. Soc. Japan \textbf{60} (2008), no.~1, 265--290.

\bibitem{Kem78}
G.~Kempf, \emph{Instability in invariant theory}, Ann. of Math. (2)
  \textbf{108} (1978), no.~2, 299--316. 

\bibitem{KN79}
G.~Kempf and L.~Ness, \emph{The length of vectors in representation spaces},
  Algebraic geometry (Proc. Summer Meeting, Univ. Copenhagen, Copenhagen,
  1978), Lecture Notes in Math., vol. 732, Springer, Berlin, 1979,
  pp.~233--243.

\bibitem{Kir84}
F.~C. Kirwan, \emph{Cohomology of quotients in symplectic and algebraic
  geometry}, Princeton University Press, 1984.

\bibitem{LeBrun88}
C.~LeBrun, \emph{Counter-examples to the generalized positive action
  conjecture}, Comm. Math. Phys. \textbf{118} (1988), no.~4, 591--596.

\bibitem{LS94}
C.~LeBrun and S.~R. Simanca, \emph{Extremal {K}\"ahler metrics and complex
  deformation theory}, Geom. and Func. Anal. \textbf{4} (1994), no.~3,
  298--336.

\bibitem{LM85}
R.~B. Lockhart and R.~C. McOwen, \emph{Elliptic differential operators on
  noncompact manifolds}, Ann. Sc. Norm. Super. Pisa Cl. Sci. (4) \textbf{1}
  (1985), no.~3, 409--447.

\bibitem{Mab08}
T.~Mabuchi, \emph{K-stability of constant scalar curvature polarization},
  arXiv:0812.4093.

\bibitem{Maz91}
R.~Mazzeo, \emph{Elliptic theory of edge operators {I}.}, Comm. in PDE
  \textbf{10} (1991), 1616--1664.

\bibitem{Mel93}
R.~Melrose, \emph{The {A}tiyah-{P}atodi-{S}inger index theorem}, Research Notes
  in Mathematics, vol.~4, A K Peters Ltd., Wellesley, MA, 1993.

\bibitem{MFK94}
D.~Mumford, J.~Fogarty, and F.~Kirwan, \emph{Geometric invariant theory}, third
  ed., Ergebnisse der Mathematik und ihrer Grenzgebiete (2) [Results in
  Mathematics and Related Areas (2)], vol.~34, Springer-Verlag, Berlin, 1994. 

\bibitem{PR00}
F.~Pacard and T.~Rivi\'ere, \emph{Linear and nonlinear aspects of vortices.
  {T}he {G}inzburg-{L}andau model}, Birkh\"auser Boston Inc., Boston, MA, 2000.

\bibitem{PS08}
D.~H. Phong and J.~Sturm, \emph{Lectures on stability and constant scalar
  curvature}, Current Developments in Mathematics 2007, International Press.

\bibitem{Sim91}
S.~R. Simanca, \emph{K\"ahler metrics of constant scalar curvature on bundles
  over ${C{\rm P}_{n-1}}$}, Math. Ann. \textbf{291} (1991), no.~2, 239--246.

\bibitem{Sto08}
J.~Stoppa, \emph{K-stability of constant scalar curvature {K}\"ahler
  manifolds}, Adv. Math. \textbf{221} (2009), no.~4, 1397--1408.

\bibitem{Sto10}
\bysame, \emph{Unstable blowups}, J. Algebraic Geom. \textbf{19} (2010), no.~1,
  1--17.

\bibitem{SSz09}
J.~Stoppa and G.~Sz\'ekelyhidi, \emph{Relative {K}-stability of extremal
  metrics}, arXiv:0912.4095.

\bibitem{GSz04}
G.~Sz\'ekelyhidi, \emph{Extremal metrics and ${K}$-stability}, Bull. Lond.
  Math. Soc. \textbf{39} (2007), no.~1, 76--84.

\bibitem{Thomas06}
R.~P. Thomas, \emph{Notes on {GIT} and symplectic reduction for bundles
  andvarieties}, Surveys in Differential Geometry, 10 (2006): A Tribute to
  Professor S.-S. Chern.

\bibitem{Tian97}
G.~Tian, \emph{K\"ahler-{E}instein metrics with positive scalar curvature},
  Invent. Math. \textbf{137} (1997), 1--37.

\bibitem{Yau93}
S.-T. Yau, \emph{Open problems in geometry}, Proc. Symposia Pure Math.
  \textbf{54} (1993), 1--28.

\end{thebibliography}
\end{document}